\documentclass[smallextended]{svjour3New}   

\usepackage{graphicx}
\usepackage{amsmath,amssymb,xcolor, enumerate,tikz}
\usepackage[hypertexnames=false]{hyperref}
\usepackage{adjustbox}
\usepackage{enumitem}

\newcommand{\intr}{\operatorname{intr}}

\newcommand{\conv}{\operatorname{conv}}
\newcommand{\GP}{\operatorname{GP}}

\newcommand{\grad}{\nabla f}
\newcommand{\R}{\mathbb{R}}
\newcommand{\Z}{\mathbb{Z}}

\newcommand{\cM}{\mathcal{M}}
\newcommand{\uSet}{\mathcal{U}}
\newcommand{\lef}{\le_f}
\newcommand{\lf}{<_f}
\newcommand{\ef}{=_f}
\newcommand{\ler}{\le_r}
\newcommand{\lr}{<_r}
\newcommand{\er}{=_r}

\newcommand{\flip}{\textsc{Flip}}

\spnewtheorem*{ex}{Example}{\bfseries}{\itshape}
\spnewtheorem{obs}{Observation}{\bfseries}{\itshape}

\title{Constructing lattice-free gradient polyhedra\\in dimension two\footnote{This paper extends work from an IPCO paper of the same title.
This version includes complete proofs as well as new results on convergence for functions that are Lipschitz continuous and strongly convex.}}
\author{Joseph Paat~$^\dagger$, Miriam Schl\"oter~$^\dagger$,\\ Emily Speakman~$^*$}
\date{}

\institute{
	$\dagger$ Department of Mathematics, ETH Z\"urich, Switzerland
	\\\email{\{joseph.paat, miriam.schloeter\}@ifor.math.ethz.ch}\\
	$*$ Department of Mathematical and Statistical Sciences, University of Colorado Denver, USA\\\email{emily.speakman@ucdenver.edu}
}

\begin{document}

\maketitle

\begin{abstract}
Lattice-free gradient polyhedra can be used to certify optimality for mixed integer convex minimization models.
We consider how to construct these polyhedra for unconstrained models with two integer variables under the assumption that all level sets are bounded.
A classic result of Bell, Doignon, and Scarf states that a lattice-free gradient polyhedron with at most four facets exists in this setting.
We present an algorithm for creating a sequence of gradient polyhedra, each of which has at most four facets, that finitely converges to a lattice-free gradient polyhedron.
Each update requires constantly many gradient evaluations.
Our updates imitate the gradient descent algorithm, and consequently, it yields a gradient descent type of algorithm for problems with two integer variables.
\end{abstract}
%
\section{Introduction.}
 \pagenumbering{arabic}

Let $f: \R^d \times \R^n \to \R$ be convex and differentiable with gradient $\grad : \R^d \times \R^n \to \R^d \times \R^n$.
We assume oracle access to $\grad$.
The unconstrained mixed integer convex minimization problem is
\begin{equation}\label{eqMainProb}\tag{CM}
\min\{f(x, z) : ~ (x, z) \in \R^d \times \Z^n\}.
\end{equation}
Applications of~\eqref{eqMainProb} include statistical regression and the closest vector problem.
In this paper we consider how to solve~\eqref{eqMainProb} by constructing an optimality certificate in the form of a lattice-free gradient polyhedron.
The \emph{gradient polyhedron corresponding to a non-empty finite set $\uSet \subseteq  \R^d \times \Z^n$} is
\[
\GP(\uSet) := \{ (x, z) \in \R^d \times \R^n :
\grad(\overline{x}, \overline{z})^\intercal ( x -\overline{x}, z - \overline{z}) \le 0 ~\forall~ (\overline{x}, \overline{z}) \in \uSet\}.
\]

We say that $\GP(\uSet)$ is lattice-free if $\intr(\GP(\uSet)) \cap \R^d \times \Z^n = \emptyset$, where
\[
\intr(\GP(\uSet)) :=  \{ (x, z) \in \R^d \times \R^n :
\grad(\overline{x}, \overline{z})^\intercal ( x -\overline{x}, z - \overline{z}) < 0 ~\forall~ (\overline{x}, \overline{z}) \in \uSet\}.
\]
This definition of interior implies that $\intr(\GP(\uSet)) = \emptyset$ if there exists $(\overline{x}, \overline{z}) \in \uSet$ such that $\grad(\overline{x}, \overline{z}) = \mathbf{0}$.
The significance of lattice-free gradient polyhedra comes from the definition of convexity:
\[
f(x,z) \ge f(\overline{x},\overline{z}) + \grad(\overline{x},\overline{z})^\intercal(x-\overline{x},z-\overline{z}) \quad \forall ~ (x,z), (\overline{x},\overline{z}) \in \R^d\times\R^n.
\]
If $\GP(\uSet)$ is lattice-free, then $\uSet$ contains a minimizer of~\eqref{eqMainProb}.

Constructing lattice-free gradient polyhedra is well studied when $n =0 $.
In fact, the classic gradient descent algorithm is a search algorithm for a lattice-free set in this setting.
See~\cite[Chapter 9]{BV2004} for a standard discussion on gradient descent.
Gradient descent updates a vector $x^i \in \R^d$ to $x^{i+1} = x^i - \alpha^i \grad(x^i)$, where $\alpha^i > 0$.
Under mild assumptions, the sequence $(x^i)_{i=0}^\infty$ converges to a vector $x^*$ with $\grad(x^*) = \mathbf{0}$.
One can verify that
\[
\grad(x^*) = \mathbf{0}  ~\text{if and only if}~\intr(\GP(\{x^*\})) = \emptyset,
\]
so $(x^i)_{i=0}^\infty$ corresponds to a sequence of gradient polyhedra $(\GP(\{x^i\}))_{i=0}^\infty$ that converges to a lattice-free gradient polyhedron $\GP(\{x^*\})$.
Notably, each gradient polyhedron that gradient descent generates has at most $2^n = 2^0 = 1$ many facets and the initial $x^0$ can be chosen arbitrarily.
Also, the algorithm only requires gradient evaluations (not even function evaluations are required).

Lattice-free gradient polyhedra can also be constructed iteratively when $n = 1$ and $d = 0$.
This certificate can be found by starting with an arbitrary $\uSet^0 = \{z^0, z^0+1\} \subseteq \Z$ and updating it as follows:
\[
\uSet^{i+1} :=
\begin{cases}
\{z^i-1, z^i\} & \text{if } 0 < \grad(z^i) \\
\{z^i+1, z^i+2\} & \text{if } \grad(z^i+1) < 0\\
\uSet^i & \text{if } \grad(z^i) \le 0 \le \grad(z^i+1).
\end{cases}
\]
The gradient comparisons ensure that $\uSet^i$ is updated by `flipping' closer to the minimizer $x^*$ of the continuous relaxation of~\eqref{eqMainProb}.
If $\uSet^{i+1} = \uSet^i$, then $\uSet^i = \{\lfloor x^*\rfloor,  \lceil x^* \rceil\}$ and $\GP(\uSet^i)$ is lattice-free.
This update procedure is not the most efficient way of obtaining $\{\lfloor x^*\rfloor,  \lceil x^* \rceil\}$ as one can simply round $x^*$.
However, the procedure does have the same properties as gradient descent: every gradient polyhedron has at most $2^n = 2^1 = 2$ many facets,  the initial set $\uSet^0$ is arbitrary, and each update only requires a constant number of gradient evaluations.
Both the update and the rounding approach generalize to when $d \ge 1$, but they do not extend naturally if $n \ge 2$.

The known results on lattice-free gradient polyhedra when $n \ge 2$ are existential rather than algorithmic.
It follows from the work of Bell, Doignon, and Scarf~\cite{B1977,D1973,S1977} that there exists $\uSet \subseteq \Z^n$ such that $|\uSet| \le 2^n$ and $\GP(\uSet)$ is lattice-free.
Baes et al.~\cite{BOW2016} extend this to lattice-free gradient polyhedra for problems with additional convex constraints.
Basu et al.~\cite{BCCWW2017} generalize gradient polyhedra further to so-called $S$-free sets.

The goal of this paper is to iteratively construct a lattice-free gradient polyhedron when $n \ge 2$.
We aim to design an algorithm that serves as a mixed integer counterpart to the gradient descent in the following sense: it creates a sequence of gradient polyhedra that each have at most $2^n$ many facets and that converges to a lattice-free polyhedron, it only requires gradient evaluations, and the initial $\uSet$ can be arbitrarily chosen.
We say such an algorithm is of \emph{gradient descent type}.

We make the following assumption:
\begin{equation}\label{eqLevelSets}
\text{$\forall ~\alpha \in \R$, the level set $\{(x,z) \in \R^d \times \R^n : f(x,z) \le \alpha\}$ is bounded.}
\end{equation}
Assumption~\eqref{eqLevelSets} guarantees that a lattice-free optimality certificate exists and that~\eqref{eqMainProb} has an optimal solution.
This assumption captures families such as strictly convex and strongly convex functions.
The techniques in this paper do not immediately extend to detect if this assumption is violated.

\subsection{Related work}

Lattice-free convex sets (not necessarily gradient polyhedra) have been heavily studied in the context of valid inequality generation for integer programs, see, e.g., \cite{ALWW2007,B1971,BCD2015,BCK2012,DW2012} and the references therein.
We also refer to~\cite{AKW2017,AWW2011,BCCZ2010,DW2008,L1989} for the structure of inclusion-wise maximal lattice-free polyhedra.
Algorithms for~\eqref{eqMainProb} use techniques such as branch and bound (see, e.g.,~\cite{GR1985,S2001}), outer approximations (see, e.g.,~\cite{BBCCGLLLMSW2008,DG1986}), convex separation~\cite[Theorem 6.7.10]{GLS1988}, or improvement oracles~\cite{BOWW2013}.
Each of these algorithms use non-gradient information or create polyhedral relaxations with potentially more than $2^n$ facets.
Baes et al.~\cite{BOWW2013} give a geometric algorithm for~\eqref{eqMainProb} when $d = 0$ and $n = 2$ but explicitly use knowledge of a bounded set containing the minimum.
In particular, they solve~\eqref{eqMainProb} by subdividing a box $B$ containing a minimizer of~\eqref{eqMainProb} and solving $O(\ln(B))$ many continuous problem.

\subsection{Statement of results}
Our main contribution is a gradient descent type algorithm when $d = 0$ and $n = 2$.
Figure~\ref{figExample} provides an example.
The algorithm creates a sequence of gradient polyhedra $(\GP(\uSet^i))_{i=0}^{\infty} $ defined by unimodular sets.
We say $\uSet \subseteq \Z^2$ is \emph{unimodular} if
\begin{equation}\label{eqUniSet}
\uSet := \uSet(z,U) := \{z + U e : e \in \{0,1\}^2\},
\end{equation}
for $z \in \Z^2$ and a matrix $U \in \Z^{2\times2}$ with $|\det(U)| = 1$.
Table~\ref{tableFlips} outlines how we update $\uSet^i$ to $\uSet^{i+1}$.
Every gradient polyhedron we generate has at most $2^n = 2^2 = 4$ many facets because it is defined by a unimodular set.
Also, the initial set $\uSet^0$ can be any arbitrary unimodular set, and each update requires a constant number of gradient evaluations.\footnote{We use `gradient evaluation' to refer a single inner product evaluation using gradients, and `constantly many' can be chosen to be 20 as counted in Table~\ref{tableFlips}.}
%
%

%
\begin{figure}
\begin{tabular}{c@{\hskip .325 in}c@{\hskip .325 in}c@{\hskip .325 in}c}
\begin{tikzpicture}[scale = .5]
\clip (-2.25,-2.25) rectangle (2.25,2.25);


\def\xstar{2};\def\ystar{2}
\def\Inter{((6*\xstar+1)*\xstar + (2*\ystar+1)*\ystar)  / (2*\ystar+1)}
\def\Slope{-(6*\xstar+1)/(2*\ystar+1)}
\draw[domain=-3:3, smooth, variable=\x, red] plot ({\x},{\x*\Slope + \Inter});

\def\xstar{1};\def\ystar{2}
\def\Inter{((6*\xstar+1)*\xstar + (2*\ystar+1)*\ystar)  / (2*\ystar+1)}
\def\Slope{-(6*\xstar+1)/(2*\ystar+1)}
\draw[domain=-3:3, smooth, variable=\x, red] plot ({\x},{\x*\Slope + \Inter});

\def\xstar{0};\def\ystar{1}
\def\Inter{((6*\xstar+1)*\xstar + (2*\ystar+1)*\ystar)  / (2*\ystar+1)}
\def\Slope{-(6*\xstar+1)/(2*\ystar+1)}
\draw[domain=-3:3, smooth, variable=\x, red] plot ({\x},{\x*\Slope + \Inter});

\def\xstar{1};\def\ystar{1}
\def\Inter{((6*\xstar+1)*\xstar + (2*\ystar+1)*\ystar)  / (2*\ystar+1)}
\def\Slope{-(6*\xstar+1)/(2*\ystar+1)}
\draw[domain=-3:3, smooth, variable=\x, red] plot ({\x},{\x*\Slope + \Inter});

\def\xstar{0};\def\ystar{1}
\def\InterStar{((6*\xstar+1)*\xstar + (2*\ystar+1)*\ystar)  / (2*\ystar+1)}
\def\SlopeStar{-(6*\xstar+1)/(2*\ystar+1)}

\def\xhat{1};\def\yhat{1}
\def\InterHat{((6*\xhat+1)*\xhat + (2*\yhat+1)*\yhat)  / (2*\yhat+1)}
\def\SlopeHat{-(6*\xhat+1)/(2*\yhat+1)}

\def\xIntersect{(\InterHat - \InterStar) / (\SlopeStar - \SlopeHat)}
\def\yIntersect{\SlopeStar*\xIntersect + \InterStar}

\draw[draw = none, fill = red!50, opacity = .3] (-3,-3) -- (-3, {-3*\SlopeStar+\InterStar}) -- ({\xIntersect}, {\yIntersect}) -- (3, {3*\SlopeHat+\InterHat}) -- cycle;

\foreach \i in {-2,...,2}{
\foreach \j in {-2,...,2}
\draw[fill = black!50, draw = black!50](\i,\j) circle (.5 ex);
}

\foreach \i in {0, 2, 4, 6,10, 16}{
\def\radx{sqrt((3*\i+1)/9)}
\def\rady{sqrt((3*\i+1)/3)}
\draw[black!50] (-1/6,-1/2) ellipse ({\radx} and {\rady} );
}

\draw[fill = black, draw = black](0,1) circle (.75 ex);
\draw[fill = black, draw = black](1,1) circle (.75 ex);
\draw[fill = black, draw = black](2,2) circle (.75 ex);
\draw[fill = black, draw = black](1,2) circle (.75 ex);
\draw[fill = red, draw = red](-1/6,-1/2) circle (.75 ex);


\draw[dashed, thick] (1,2)--(0,1)--(1,1)--(2,2)--cycle;

\end{tikzpicture}


&

\begin{tikzpicture}[scale = .5]
\clip (-2.25,-2.25) rectangle (2.25,2.25);


\def\xstar{-1};\def\ystar{0}
\def\Inter{((6*\xstar+1)*\xstar + (2*\ystar+1)*\ystar)  / (2*\ystar+1)}
\def\Slope{-(6*\xstar+1)/(2*\ystar+1)}
\draw[domain=-3:3, smooth, variable=\x, red] plot ({\x},{\x*\Slope + \Inter});

\def\xstar{0};\def\ystar{1}
\def\Inter{((6*\xstar+1)*\xstar + (2*\ystar+1)*\ystar)  / (2*\ystar+1)}
\def\Slope{-(6*\xstar+1)/(2*\ystar+1)}
\draw[domain=-3:3, smooth, variable=\x, red] plot ({\x},{\x*\Slope + \Inter});

\def\xstar{-1};\def\ystar{1}
\def\Inter{((6*\xstar+1)*\xstar + (2*\ystar+1)*\ystar)  / (2*\ystar+1)}
\def\Slope{-(6*\xstar+1)/(2*\ystar+1)}
\draw[domain=-3:3, smooth, variable=\x, red] plot ({\x},{\x*\Slope + \Inter});

\def\xstar{-2};\def\ystar{0}
\def\Inter{((6*\xstar+1)*\xstar + (2*\ystar+1)*\ystar)  / (2*\ystar+1)}
\def\Slope{-(6*\xstar+1)/(2*\ystar+1)}
\draw[domain=-3:3, smooth, variable=\x, red] plot ({\x},{\x*\Slope + \Inter});

\def\xstar{-1};\def\ystar{0}
\def\InterStar{((6*\xstar+1)*\xstar + (2*\ystar+1)*\ystar)  / (2*\ystar+1)}
\def\SlopeStar{-(6*\xstar+1)/(2*\ystar+1)}

\def\xhat{0};\def\yhat{1}
\def\InterHat{((6*\xhat+1)*\xhat + (2*\yhat+1)*\yhat)  / (2*\yhat+1)}
\def\SlopeHat{-(6*\xhat+1)/(2*\yhat+1)}

\def\xIntersect{(\InterHat - \InterStar) / (\SlopeStar - \SlopeHat)}
\def\yIntersect{\SlopeStar*\xIntersect + \InterStar}

\draw[draw = none, fill = red!50, opacity = .3] (-3, {-3*\SlopeStar+\InterStar}) -- ({\xIntersect}, {\yIntersect}) -- (3, {3*\SlopeHat+\InterHat}) --(3,-3)-- cycle;

\foreach \i in {-2,...,2}{
\foreach \j in {-2,...,2}
\draw[fill = black!50, draw = black!50](\i,\j) circle (.5 ex);
}

\foreach \i in {0, 2, 4, 6,10, 16}{
\def\radx{sqrt((3*\i+1)/9)}
\def\rady{sqrt((3*\i+1)/3)}
\draw[black!50] (-1/6,-1/2) ellipse ({\radx} and {\rady} );
}

\draw[fill = black, draw = black](0,1) circle (.75 ex);
\draw[fill = black, draw = black](-1,1) circle (.75 ex);
\draw[fill = black, draw = black](-2,0) circle (.75 ex);
\draw[fill = black, draw = black](-1,0) circle (.75 ex);

\draw[fill = red, draw = red](-1/6,-1/2) circle (.75 ex);


\draw[dashed, thick] (-1,0)--(0,1)--(-1,1)--(-2,0)--cycle;

\end{tikzpicture}


&

\begin{tikzpicture}[scale = .5]
\clip (-2.25,-2.25) rectangle (2.25,2.25);


\def\xOne{-1};\def\yOne{-1}
\def\InterOne{((6*\xOne+1)*\xOne + (2*\yOne+1)*\yOne)  / (2*\yOne+1)}
\def\SlopeOne{-(6*\xOne+1)/(2*\yOne+1)}
\draw[domain=-3:3, smooth, variable=\x, red] plot ({\x},{\x*\SlopeOne + \InterOne});

\def\xTwo{-1};\def\yTwo{0}
\def\InterTwo{((6*\xTwo+1)*\xTwo + (2*\yTwo+1)*\yTwo)  / (2*\yTwo+1)}
\def\SlopeTwo{-(6*\xTwo+1)/(2*\yTwo+1)}
\draw[domain=-3:3, smooth, variable=\x, red] plot ({\x},{\x*\SlopeTwo + \InterTwo});

\def\xThree{0};\def\yThree{1}
\def\InterThree{((6*\xThree+1)*\xThree + (2*\yThree+1)*\yThree)  / (2*\yThree+1)}
\def\SlopeThree{-(6*\xThree+1)/(2*\yThree+1)}
\draw[domain=-3:3, smooth, variable=\x, red] plot ({\x},{\x*\SlopeThree + \InterThree});

\def\xFour{0};\def\yFour{0}
\def\InterFour{((6*\xFour+1)*\xFour + (2*\yFour+1)*\yFour)  / (2*\yFour+1)}
\def\SlopeFour{-(6*\xFour+1)/(2*\yFour+1)}
\draw[domain=-3:3, smooth, variable=\x, red] plot ({\x},{\x*\SlopeFour + \InterFour});

\def\xOneTwo{(\InterOne - \InterTwo) / (\SlopeTwo - \SlopeOne)}
\def\yOneTwo{\SlopeOne*\xOneTwo + \InterOne}

\def\xTwoFour{(\InterFour - \InterTwo) / (\SlopeTwo - \SlopeFour)}
\def\yTwoFour{\SlopeFour *\xTwoFour  + \InterFour }

\draw[draw = none, fill = red!50, opacity = .3] ({0}, {0*\SlopeOne+\InterOne}) -- ({\xOneTwo}, {\yOneTwo}) --
																	({\xTwoFour}, {\yTwoFour}) -- ({3},{3*\SlopeFour+\InterFour}) -- cycle;

\foreach \i in {-2,...,2}{
\foreach \j in {-2,...,2}
\draw[fill = black!50, draw = black!50](\i,\j) circle (.5 ex);
}

\foreach \i in {0, 2, 4, 6,10, 16}{
\def\radx{sqrt((3*\i+1)/9)}
\def\rady{sqrt((3*\i+1)/3)}
\draw[black!50] (-1/6,-1/2) ellipse ({\radx} and {\rady} );
}

\draw[fill = black, draw = black](0,0) circle (.75 ex);
\draw[fill = black, draw = black](0,1) circle (.75 ex);
\draw[fill = black, draw = black](-1,0) circle (.75 ex);
\draw[fill = black, draw = black](-1,-1) circle (.75 ex);

\draw[fill = red, draw = red](-1/6,-1/2) circle (.75 ex);


\draw[dashed, thick] (-1,0)--(0,1)--(0,0)--(-1,-1)--cycle;

\end{tikzpicture}


&

\begin{tikzpicture}[scale = .5]
\clip (-2.25,-2.25) rectangle (2.25,2.25);


\def\xOne{-1};\def\yOne{0}
\def\InterOne{((6*\xOne+1)*\xOne + (2*\yOne+1)*\yOne)  / (2*\yOne+1)}
\def\SlopeOne{-(6*\xOne+1)/(2*\yOne+1)}
\draw[domain=-3:3, smooth, variable=\x, red] plot ({\x},{\x*\SlopeOne + \InterOne});

\def\xTwo{0};\def\yTwo{0}
\def\InterTwo{((6*\xTwo+1)*\xTwo + (2*\yTwo+1)*\yTwo)  / (2*\yTwo+1)}
\def\SlopeTwo{-(6*\xTwo+1)/(2*\yTwo+1)}
\draw[domain=-3:3, smooth, variable=\x, red] plot ({\x},{\x*\SlopeTwo + \InterTwo});

\def\xThree{-1};\def\yThree{-1}
\def\InterThree{((6*\xThree+1)*\xThree + (2*\yThree+1)*\yThree)  / (2*\yThree+1)}
\def\SlopeThree{-(6*\xThree+1)/(2*\yThree+1)}
\draw[domain=-3:3, smooth, variable=\x, red] plot ({\x},{\x*\SlopeThree + \InterThree});

\def\xFour{0};\def\yFour{-1}
\def\InterFour{((6*\xFour+1)*\xFour + (2*\yFour+1)*\yFour)  / (2*\yFour+1)}
\def\SlopeFour{-(6*\xFour+1)/(2*\yFour+1)}
\draw[domain=-3:3, smooth, variable=\x, red] plot ({\x},{\x*\SlopeFour + \InterFour});

\def\xOneThree{(\InterOne - \InterThree) / (\SlopeThree - \SlopeOne)}
\def\yOneThree{\SlopeOne*\xOneThree + \InterOne}

\def\xOneTwo{(\InterOne - \InterTwo) / (\SlopeTwo - \SlopeOne)}
\def\yOneTwo{\SlopeOne*\xOneTwo + \InterOne}

\def\xTwoFour{(\InterFour - \InterTwo) / (\SlopeTwo - \SlopeFour)}
\def\yTwoFour{\SlopeFour *\xTwoFour  + \InterFour }

\def\xThreeFour{(\InterFour - \InterThree) / (\SlopeThree - \SlopeFour)}
\def\yThreeFour{\SlopeFour *\xThreeFour  + \InterFour }

\draw[draw = none, fill = red!50, opacity = .3] ({\xOneThree}, {\yOneThree}) -- ({\xThreeFour}, {\yThreeFour}) --
																	({\xTwoFour}, {\yTwoFour}) -- ({\xOneTwo},{\yOneTwo}) -- cycle;

\foreach \i in {-2,...,2}{
\foreach \j in {-2,...,2}
\draw[fill = black!50, draw = black!50](\i,\j) circle (.5 ex);
}

\foreach \i in {0, 2, 4, 6,10, 16}{
\def\radx{sqrt((3*\i+1)/9)}
\def\rady{sqrt((3*\i+1)/3)}
\draw[black!50] (-1/6,-1/2) ellipse ({\radx} and {\rady} );
}

\draw[fill = black, draw = black](0,0) circle (.75 ex);
\draw[fill = black, draw = black](0,-1) circle (.75 ex);
\draw[fill = black, draw = black](-1,-1) circle (.75 ex);
\draw[fill = black, draw = black](-1,0) circle (.75 ex);

\draw[fill = red, draw = red](-1/6,-1/2) circle (.75 ex);


\draw[dashed, thick] (0,0)--(0,-1)--(-1,-1)--(-1,0)--cycle;

\end{tikzpicture}

\\
$\uSet^0$ & $\uSet^{1}$ & $ \uSet^{2}$ & $\uSet^3$
\end{tabular}
\caption{A sequence $(\uSet^i)_{i=0}^3$ of unimodular sets generated by our updates for $f(x_1, x_2) := 3x_1^2 + x_2^2 + x_1+x_2$.
The convex hull of each $\uSet^i$ is outlined in black, and level curves of $f$ are in gray.
The hyperplanes defining $\GP(\uSet^i)$ are in red, the minimizer of the continuous relaxation of~\eqref{eqMainProb} is drawn as a red point, and $\GP(\uSet^i)$ is shaded in red.}\label{figExample}
\end{figure}
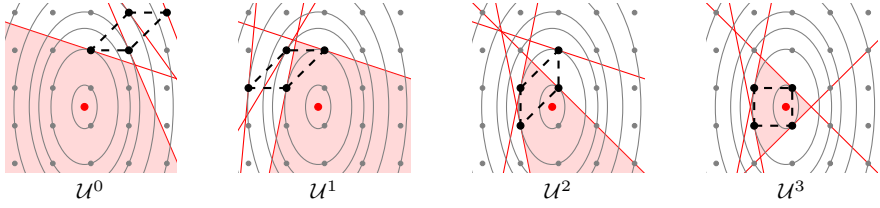

We use two measures of progress to ensure that $(\uSet^i)_{i=0}^{\infty} $ converges to a lattice-free gradient polyhedron $\GP(\uSet^T)$.
The first measure is the minimum function value in $\uSet^i$:
\begin{equation}\label{eqFunctionVal}
\min\{f(z) : z \in \uSet^i\}.
\end{equation}
The second measure is the distance from the set of optimal solution of~\eqref{eqMainProb} to $\uSet^i$ with respect to $U^i$:
\begin{equation}\label{eqDistVal}
\min\{\|(U^i)^{-1}(z^* - z)\|_1 : z \in \uSet^i \text{ and } z^* \text{ is optimal for}~\eqref{eqMainProb}\}.
\end{equation}
We discuss~\eqref{eqDistVal} more in Section~\ref{secFlippingAlg}.
Our updates are such that~\eqref{eqFunctionVal} and~\eqref{eqDistVal} are both \emph{non-increasing} in $i$.
Furthermore, if $\uSet^{i}$ does not contain a minimizer of~\eqref{eqMainProb}, then at least one measure strictly decreases after at most two updates.
The unimodular set requires preprocessing, which is described in Lemma~\ref{eqPreprocess}.

\begin{theorem}\label{thmMain1}
Let $\uSet^i $ be a preprocessed unimodular set.
If $\uSet^i $ does not fit into Table~\ref{tableFlips}, then it is lattice-free.
Otherwise, it can be updated to a unimodular set $\uSet^{i+1}$ such that
\begin{enumerate}[leftmargin = .75 cm, label = \textit{(\roman*)}]
\item neither \eqref{eqFunctionVal} nor~\eqref{eqDistVal} increases from $\uSet^i$ to $\uSet^{i+1}$,
\smallskip
\item if \eqref{eqFunctionVal} and~\eqref{eqDistVal} remain constant and $\uSet^{i+1}$ is updated to $\uSet^{i+2}$, then
\smallskip
\begin{enumerate}[leftmargin = .75 cm, label = (ii-\alph*)]
\item \eqref{eqFunctionVal} or~\eqref{eqDistVal} strictly decreases from $\uSet^{i+1}$ to $\uSet^{i+2}$ or
\item $\uSet^{i}$ contains a minimizer of~\eqref{eqMainProb}.
\end{enumerate}
\end{enumerate}
Moreover, this can all be checked using constantly many gradient evaluations.
\end{theorem}

Assumption~\eqref{eqLevelSets} implies that~\eqref{eqFunctionVal} and~\eqref{eqDistVal} can only strictly decrease a finite number of times.
Therefore, if $(\uSet^i)_{i=0}^\infty$ is constructed using Table~\ref{tableFlips}, then $\GP(\uSet^i)$ is guaranteed to eventually contain a minimizer of~\eqref{eqMainProb}.
%
%
The updates also guarantee that $\GP(\uSet^i)$ eventually becomes lattice-free.
%

%
\begin{theorem}\label{thmMain2}
Let $\uSet^0$ be unimodular and $(\uSet^i)_{i=0}^\infty$ be created using Table~\ref{tableFlips}.
For some $T$ the gradient polyhedron $\GP(\uSet^T)$ is lattice-free, and this can be checked using constantly many gradient evaluations.
\end{theorem}

Our updates are conservative, but if $f$ is $L$-Lipschitz continuous and $c$-strongly convex, then only $(6L/c+2) \cdot \|z^*\|_1$ many updates are needed to find an optimal solution $z^*$ (see Proposition~\ref{corConvergence}).
%
Furthermore, our updates can be extended to $d\geq 0$ provided we are able to exactly minimize $f_z(x):= f(x,z)$ over $x\in\R^d$ for each fixed $z \in \Z^2$.
To see this, note that~\eqref{eqMainProb} is the same as minimizing $f^{\min}(z) := \min\{f_z(x):x \in \R^d\}$ over $\Z^2$.

\begin{corollary}\label{cor:ExactMixedInteger}
The updates in Table~\ref{tableFlips} extend to a gradient descent type algorithm when $n = 2$ and $d \ge 0$ if we can exactly minimize $f_z(x)$ over $\R^d$.
\end{corollary}

For the remainder of the paper, we assume $n = 2$ and $d = 0$.
If a unimodular set $\uSet^{i}$ contains a vector $z^*\in \Z^2$ such that $\grad(z^*) = \mathbf{0}$, then $\GP(\uSet^i)$ is lattice-free and $z^*$ is a minimizer of~\eqref{eqMainProb}.
We always assume that this check is made and only consider updating $\uSet^i$ if the check fails.
%

\section{Preliminary results on convexity and gradient polyhedra corresponding to unimodular sets.}

We refer to~\cite{BV2004} for more on convexity and gradients.
Given $z, \overline{z} \in \R^2$, we say that $\overline{z}$ \emph{cuts} $z$ if $\grad(\overline{z})^\intercal(z-\overline{z}) \ge 0$ and \emph{strictly cuts} $z$ if $\grad(\overline{z})^\intercal(z-\overline{z}) > 0$.
The next result follows from the definition of convexity.
\begin{proposition}\label{propConvexProperties}
Let $z, \overline{z} \in \R^2$.
If $\overline{z}$ cuts $z$ (respectively, strictly cuts), then $f(z) \ge f(\overline{z})$ (respectively, $f(z) > f(\overline{z})$).
\end{proposition}

We denote the $i$-th column of $U\in \Z^{2\times 2}$ by $u^i$.
The following result follows from Proposition~\ref{propConvexProperties} and the definition of gradient polyhedra.
%
\begin{lemma}\label{Lemma:Gen_Gradient_Polyhedron}
	Let $\mathcal{X} \subseteq \Z^2$ be a non-empty finite set and $\uSet = \uSet(z,U)$ be unimodular.
The following hold:
	\begin{enumerate}[leftmargin = .75 cm, label = \textit{(\roman*)}]
		\item \label{Prop:partI} If $z \not\in \intr (\GP(\mathcal{X} ))$, then $f(z) \geq \min\{f(\overline{z}): \overline{z} \in \mathcal{X} \}$.
		\item \label{Prop:partII} ${\displaystyle\mathcal{X}  \cap \GP(\mathcal{X} )} \neq \emptyset$.
				%
		%
%
		\item \label{Prop:partIII}If $\uSet \cap \GP(\uSet) = \{z, z+u^1+u^2\}$, then $z+u^i$ does not strictly cut any vector in $\uSet$ for each $i \in \{1,2\}$.
		\item \label{Prop:partIV}If $z$ cuts $z+u^1$ and $z+u^2$, then $z$ strictly cuts $z+u^1+u^2$.
		\item \label{Prop:partV} If $z \in \GP(\uSet)$ and no vector in $\uSet$ is strictly cut by $z$, then $|\uSet \cap \GP(\uSet)| \ge 2$.
	\end{enumerate}
\end{lemma}

%

Let $\uSet = \uSet(z, U)$ be unimodular.
After multiplying $u^1$ and $u^2$ by $ \pm 1$ and relabeling the `anchor' vector $z$ to be another vector in $\uSet$, we assume $\GP(\uSet)$ fulfills \emph{preprocessing} properties.
\begin{lemma}\label{eqPreprocess}
Let $\uSet = \uSet(z,U)$ be unimodular.
	We can preprocess $\GP(\uSet)$ so that
	\begin{enumerate}[leftmargin = .75 cm, label = \textit{(\roman*)}]
	\item\label{PreProcess1}  $z \in \GP(\uSet)$,
	\item\label{PreProcess2} if $|\uSet \cap \GP(\uSet)| =  2$, then
	\begin{enumerate}[leftmargin = .75 cm, label = \textit{(ii-\alph*)}]
	\item\label{PreProcess2a} $\uSet \cap \GP(\uSet) =  \{z, z+u^1\}$, or
	\item\label{PreProcess2b} $\uSet \cap \GP(\uSet) =  \{z, z+u^1+u^2\}$, $z$ strictly cuts $z+u^1$, and $z+u^1+u^2$ strictly cuts $z+u^2$.
	\end{enumerate}
	\item\label{PreProcess3} if $|\uSet \cap \GP(\uSet)| =  3$, then $\uSet \cap \GP(\uSet) = \{z, z+u^1, z+u^2\}$.
	\end{enumerate}
\end{lemma}
\begin{proof}
Lemma~\ref{Lemma:Gen_Gradient_Polyhedron}~\ref{Prop:partII} states $\uSet \cap \GP(\uSet) \neq \emptyset$.
Thus, we can relabel $\uSet$ so that $z \in \uSet \cap \GP(\uSet)$ and~\ref{PreProcess1} holds.
For~\ref{PreProcess2} and~\ref{PreProcess3}, we refer to the following figure:
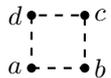
\begin{center}
		\begin{tikzpicture}[scale = .7]
			\draw[dashed, thick](2,2)--(3,2)--(3,3)--(2,3)--cycle;
			\draw[fill = black, draw = black](2,2) circle (.5 ex) node[left]{$a$};
			\draw[fill = black, draw = black](3,2) circle (.5 ex) node[right]{$b$};
			\draw[fill = black, draw = black](3,3) circle (.5 ex) node[right]{$c$};
			\draw[fill = black, draw = black](2,3) circle (.5 ex) node[left]{$d$};
		\end{tikzpicture}
\end{center}
We have labeled the vectors in $\uSet$ as $a,b,c,d$.
The vectors $b-a = c-d$ and $d-a = c-b$ are either $u^1$ or $u^2$ or a negative thereof.
After possibly multiplying $u^1$ or $u^2$ by $-1$, any three vectors in $\uSet$ can be relabeled to be $z,z+u^1$ and $z+u^2$.
Thus, if $|\uSet \cap \GP(\uSet)|  = 3$, then $\uSet$ can be relabeled to satisfy~\ref{PreProcess3}.
Similarly, if $\uSet \cap \GP(\uSet)$ is equal to $\{a,b\},\{a,d\},\{b,c\}$, or $\{c,d\}$, then $\uSet$ can be relabeled so that~\ref{PreProcess2a} holds.

It is left to consider if $\uSet \cap \GP(\uSet)$ is equal to $\{a,c\}$ or $\{b,d\}$.
Without loss of generality suppose $\uSet \cap \GP(\uSet) = \{a,c\}$.
Relabel $a$ to be $z$, $c$ to be $z+u^1+u^2$, $b$ to be $z+u^1$ and $d$ to be $z+u^2$.
	Lemma~\ref{Lemma:Gen_Gradient_Polyhedron} \ref{Prop:partIII} and \ref{Prop:partIV} imply that $z$ and $z+u^1+u^2$ each strictly cut exactly one of $z+u^1$ and $z+u^2$.
	After possibly relabeling one last time, we may assume that $z$ strictly cuts $z+u^1$ and $z+u^1+u^2$ strictly cuts $z+u^2$.
This proves~\ref{PreProcess2b}.\qed
\end{proof}

We end the section with an observation.

\begin{obs}\label{Observation:interior}
Let $\mathcal{X} \subseteq \Z^2$ be a finite set.
	If $x_1\in \intr(\GP(\mathcal{X}))$, $x_2\dotsc, x_t \in \GP(\mathcal{X})$, $\lambda_1, \dotsc, \lambda_t  > 0$ and $\sum_{i=1}^t \lambda_i =1$, then $\sum_{i=1}^t \lambda_i x_i \in \intr(\GP(\mathcal{X}))$.
\end{obs}

\section{Updating gradient polyhedra when $n = 2$ and $d  = 0$.}\label{secFlippingAlg}
%
Let $\uSet = \uSet(z, U)$ be a unimodular set that satisfies the preprocessing of Lemma~\ref{eqPreprocess}.
We update $\uSet$ by replacing, or `flipping', $U$ with a matrix $\overline{U}$ and then preprocessing $(z, \overline{U})$.
We use $\flip(U)$ to denote the updated matrix $\overline{U}$, and $\flip(\uSet)$ to denote the unimodular set obtained after preprocessing $(z, \overline{U})$.
Table~\ref{tableFlips} defines $\flip(U)$.
Certain flips rely on the following line segments:
\begin{align}
\label{eqH}
	H^i  := \{ z+k  u^1 + i  u^2 \in \intr(\GP(\uSet)) : k \in \R\} \quad \forall ~i \in \{-1,1\}.
\end{align}

 \begin{table}[htbp]
 \begin{tabular}{|@{\hskip 0 cm}l@{\hskip -.1 cm}|}
 \hline\\[-.15 cm]
\begin{tabular}{l@{\hskip .1 cm}l}
\textbf{Case \hypertarget{case1}{1}:} & $\uSet \cap \GP(\uSet) = \{z\}$.\\[.1 cm]
&For each $i \in \{1,2\}$ define $\sigma^i:= 1$ if $\grad(z)^\intercal u^i\le 0$ and $\sigma^i := -1$ otherwise.\\[.1 cm]
%
%
&Set $\flip(U) := (\sigma^1 u^1, \sigma^2 u^2)$.
\end{tabular}

\\[.375 cm]
\hspace{.375 in}
\begin{tabular}{c@{\hskip .35 cm }c@{\hskip .35 cm }c}

		\begin{tikzpicture}[baseline = 10 ex, scale = .7]

			\draw[draw = none, fill = red!30, opacity = .3](1,1)--(1,2+1/2)--(1+1/3, 2+2/3)--(2+2/3, 1+1/3)--(2+1/2,1)--cycle;

			\foreach \i in {1,...,3}{
			\foreach \j in {1,...,3}
			\draw[fill = black!50, draw = black!50](\i,\j) circle (.25 ex);
			}

			\draw[red, thick](1,3)--(3,1); 
			\draw[red, thick](1,2+1/2)--(2.25,3.125); 
			\draw[red, thick](2.75,3.25)--(3.25,2.75); 
			\draw[red, thick](2+1/2,1)--(3.25,2.5); 

			\draw[dashed, thick](2,2)--(3,2)--(3,3)--(2,3)--cycle;
			\draw[draw = black, fill = black](2,2) circle (.5 ex) node[below  left]{$z$};
			\draw[draw = black, fill = black](3,2) circle (.5 ex) node[right , align = left]{$z+u^1$};
			\draw[draw = black, fill = black](2,3) circle (.5 ex) node[above]{$z+u^2$};
			\draw[draw = black, fill = black](3,3) circle (.5 ex);

		\end{tikzpicture}
&
{\large
$\xrightarrow{\text{Update to }\overline{U} := \flip(U)}$
}
&
		\begin{tikzpicture}[baseline = 8 ex, scale = .7]

			\foreach \i in {1,...,3}{
			\foreach \j in {1,...,3}
			\draw[fill = black!50, draw = black!50](\i,\j) circle (.25 ex);
			}

			\draw[fill = white, draw = none](3,1) circle (.5 ex);
			
			\draw[dashed, thick, black!25](2,2)--(3,2)--(3,3)--(2,3)--cycle;
			\draw[dashed, thick, black](2,2)--(1,2)--(1,1)--(2,1)--cycle;
			
			\draw[draw = black, fill = black](2,2) circle (.5 ex) node[above right]{$z$};
			\draw[draw = black, fill = black](1,2) circle (.5 ex) node[left, align = left]{$z+\overline{u}^1$};
			\draw[draw = black, fill = black](2,1) circle (.5 ex) node[right]{$z+\overline{u}^2$};
			\draw[draw = black, fill = black](1,1) circle (.5 ex);

		\end{tikzpicture}
\end{tabular}

\\[.875 cm]

\hline\\[-.15 cm]

\begin{tabular}{l@{\hskip .1 cm}l}
\textbf{Case \hypertarget{case2}{2}:} & $\uSet \cap \GP(\uSet) = \{z, z+u^1+u^2\}$.\\[.1 cm]
& Define $U' := (u^1, u^1 + u^2)$, $\uSet' := \uSet(z,U')$,\\[.1 cm]
& and $U'' := (-u^1, u^1 + u^2)$, $\uSet'' := \uSet(z,U'')$. \\ [.15 cm]
&\textbf{If} $\uSet'$ is connected \textbf{or} $\GP(\uSet') \cap \uSet' = \{{z}+2u^1+u^2\}$ \textbf{or}\\[.15 cm]
&~~~~~$z+u^1+u^2$ strictly cuts $z-u^1$ and $z-u^1$ strictly cuts $z$,\\[.15 cm]
&~~~~~\textbf{then} set $\flip(U) := U'$.\\[.15 cm]
& \textbf{Else if }$\uSet''$ is connected \textbf{or} $\GP(\uSet'') \cap \uSet'' = \{z-u^1\}$ \textbf{or}\\[.15 cm]
 &~~~~~$z$ strictly cuts $z+2u^1+u^2$ and $z+2u^1+u^2$ strictly cuts $z+u^1+u^2$,\\[.15 cm]
&~~~~~\textbf{then} set $\flip(U) := U''$. \\[.15 cm]
&\textbf{Else} set $\flip(U) := (-u^1, 2u^1+u^2)$
\end{tabular}

\\[.5 cm]
\hspace{.35 in}
\begin{tabular}{@{\hskip .5 cm } c@{\hskip .35 cm } c@{\hskip .35 cm } c}

		\begin{tikzpicture}[baseline = 12 ex, scale = .7]

			\draw[draw = none, fill = red!30, opacity = .3](1,1.5)--(1, 2.25)--(3,3)--(3.25, 3)--(3.25,2.625)--cycle;

			\foreach \i in {1,...,3}{
			\foreach \j in {2,...,3}
			\draw[fill = black!50, draw = black!50](\i,\j) circle (.25 ex);
			}
			
			\draw[red, thick](1,1.5)--(3.25,2.625); 
			\draw[red, thick](1,2+1/2)--(2.25,3.125); 
			\draw[red, thick](1,2.25)--(3,3); 
			\draw[red, thick](2.5,1.625)--(3.25,2.125); 

			\draw[dashed, thick](2,2)--(3,2)--(3,3)--(2,3)--cycle;
			\draw[draw = black, fill = black](2,2) circle (.5 ex) node[below ]{$z$};
			\draw[draw = black, fill = black](3,2) circle (.5 ex) node[right , align = left]{$z+u^1$};
			\draw[draw = black, fill = black](2,3) circle (.5 ex) node[above]{$z+u^2$};
			\draw[draw = black, fill = black](3,3) circle (.5 ex);

		\end{tikzpicture}
&
{\large
$\xrightarrow{\text{Update to } \overline{U} : = \flip(U)}$
}
&
	\begin{tikzpicture}[baseline = 12 ex, scale = .7]

			\foreach \i in {1,...,3}{
			\foreach \j in {2,...,3}
			\draw[fill = black!50, draw = black!50](\i,\j) circle (.25 ex);
			}

			\draw[dashed, draw = black!25, thick](2,2)--(3,2)--(3,3)--(2,3)--cycle;
			\draw[dashed, draw = black, thick](2,2)--(1,2)--(2,3)--(3,3)--cycle;
			\draw[draw = black, fill = black](2,2) circle (.5 ex) node[below ]{$z$};
			\draw[draw = black, fill = black](1,2) circle (.5 ex) node[left , align = left]{$z+\overline{u}^1$};
			\draw[draw = black, fill = black](2,3) circle (.5 ex);
			\draw[draw = black, fill = black](3,3) circle (.5 ex) node[above right]{$z+\overline{u}^2$};

		\end{tikzpicture}
\end{tabular}

\\[.75 cm]

\hline\\[-.15 cm]

%
\begin{tabular}{l@{\hskip .1 cm}l}
\textbf{Case \hypertarget{case3}{3}:} & $\uSet \cap \GP(\uSet) = \{z, z+u^1\}$ and $|H^i \cap \Z^2| = 1$ for some $i \in \{-1,1\}$. \\[.15 cm]
& Let $z+ku^1+iu^2 \in H^i \cap \Z^2$.\\[.15 cm]
& Set $\flip(U) := (ku^1 + iu^2, -(k-1)u^1 - iu^2)$.
\end{tabular}

\\[0.5cm]
\hline\\[-.15 cm]

\begin{tabular}{l@{\hskip .1 cm}l}
\textbf{Case \hypertarget{case4}{4}:} & $\uSet \cap \GP(\uSet) =  \{z, z+u^1\}$ and $|H^i \cap \Z^2| \ge 2$ for some $i \in \{-1,1\}$.\\[.15 cm]
& Let $z+ku^1 + iu^2\in H^i \cap \Z^2$ minimize $|k|$.\\[.15 cm]
&\textbf{If} $k \ge 0$,  \textbf{then} set $\flip(U) := (u^1, ku^1 + iu^2)$.\\[.15 cm]
&\textbf{If} $k \le -1$, \textbf{then} set $\flip(U) := (u^1, (k-1)u^1 + iu^2)$.
\end{tabular}

\\[.5 cm]
\hspace{.25 in}
\begin{tabular}{@{\hskip .5 cm } c@{\hskip .35 cm } c@{\hskip .35 cm } c}
		\begin{tikzpicture}[baseline = 12 ex, scale = .7]

			\draw[draw = none, fill = red!30, opacity = .3](5.25, 3)--(1+7/8,1.5) -- (1.75,1.75)--(3.5,3) -- cycle;

			\foreach \i in {2,...,5}{
			\foreach \j in {3}
			\draw[fill = black!50, draw = black!50](\i,\j) circle (.25 ex);
			}

			\draw[red, thick](1.75,1.75)--(3.5,3); 
			\draw[red, thick](2.15, 3.15)--(1.85, 2.85); 
			\draw[red, thick](3.1,3.05)--(2,2.5); 
			\draw[red, thick](5.25, 3)--(1+7/8,1.5); 

			\draw[draw = black, fill = black](2,2) circle (.5 ex) node[ left]{$z$};
			\draw[draw = black, fill = black](3,2) circle (.5 ex) node[right, align = left]{$z+u^1$};
			\draw[draw = black, fill = black](2,3) circle (.5 ex) node[above]{$z+u^2$};
			\draw[draw = black, fill = black](3,3) circle (.5 ex) node[above right]{$z+u^1+u^2$};

			\draw[dashed, thick](2,2)--(3,2)--(3,3)--(2,3)--cycle;

		\end{tikzpicture}
&
{\large
$\xrightarrow{\text{Update to } \overline{U}:= \flip(U)}$
}
&
		\begin{tikzpicture}[baseline = 12 ex, scale = .7]

			\foreach \i in {2,...,5}{
			\foreach \j in {2,...,3}
			\draw[fill = black!50, draw = black!50](\i,\j) circle (.25 ex);
			}

			\draw[draw = black, fill = black](2,2) circle (.5 ex) node[ left]{$z$};
			\draw[draw = black, fill = black](3,2) circle (.5 ex) node[below right = 0]{$z+\overline{u}^1$};
			\draw[draw = black, fill = black](4,3) circle (.5 ex) node[above]{$z+\overline{u}^2$};
			\draw[draw = black, fill = black](5,3) circle (.5 ex);

			\draw[dashed, draw = black!25, thick](2,2)--(3,2)--(3,3)--(2,3)--cycle;
			\draw[dashed, draw = black, thick](2,2)--(4,3)--(5,3)--(3,2)--cycle;

		\end{tikzpicture}
\end{tabular}

\\[.75 cm]
\hline\\[-.15 cm]

\begin{tabular}{l@{\hskip .15 cm}l}
\textbf{Case \hypertarget{case5}{5} :} & $\uSet \cap \GP(\uSet) =  \{z, z+u^1,z+u^2\}$ and $|H^i \cap \Z^2| \geq 1 $ for some $i \in \{-1,1\}$.\\[.15 cm]
& \textbf{If} $ i = 1$, \textbf{then} set $\flip(U) := (u^1, -u^1 + u^2)$.\\ [.15 cm]
& \textbf{Else} set $\flip(U) := (u^1 - u^2, u^2)$.
\end{tabular}

\\[.5 cm]
\hspace{.25 in}
\begin{tabular}{@{\hskip .5 cm } c@{\hskip .35 cm } c@{\hskip .35 cm } c}
		\begin{tikzpicture}[baseline = 10 ex, scale = .7]

			\draw[draw = none, fill = red!30, opacity = .3](2,1)--(2,3)--(2.4, 3.2)--(3.5,1)--cycle;

			\foreach \i in {2,...,3}{
			\foreach \j in {1,...,3}
			\draw[fill = black!50, draw = black!50](\i,\j) circle (.25 ex);
			}

			\draw[red, thick](2,3)--(2,1); 
			\draw[red, thick](1,2+1/2)--(2.4, 3.2); 
			\draw[red, thick](2.75,3.25)--(3.25,2.75); 
			\draw[red, thick](2.4, 3.2)--(3.5,1); 

			\draw[draw = black, fill = black](2,2) circle (.5 ex) node[below  left]{$z$};
			\draw[draw = black, fill = black](3,2) circle (.5 ex) node[right , align = left]{$z+u^1$};
			\draw[draw = black, fill = black](2,3) circle (.5 ex) node[above]{$z+u^2$};
			\draw[draw = black, fill = black](3,3) circle (.5 ex);

			\draw[dashed, thick](2,2)--(3,2)--(3,3)--(2,3)--cycle;

		\end{tikzpicture}
&
{\large
$\xrightarrow{\text{Update to } \overline{U}:= \flip(U)}$
}
&
		\begin{tikzpicture}[baseline = 10 ex, scale = .7]

			\foreach \i in {2,...,3}{
			\foreach \j in {1,...,3}
			\draw[fill = black!50, draw = black!50](\i,\j) circle (.25 ex);
			}

			\draw[dashed, draw = black!25, thick](2,2)--(3,2)--(3,3)--(2,3)--cycle;
			\draw[draw = black, fill = black](2,2) circle (.5 ex) node[below  left]{$z$};
			\draw[draw = black, fill = black](3,1) circle (.5 ex) node[right = 0]{$z+\overline{u}^2$};
			\draw[draw = black, fill = black](2,3) circle (.5 ex) node[above]{$z+\overline{u}^1$};
			\draw[draw = black, fill = black](3,2) circle (.5 ex);

			\draw[dashed, thick](2,2)--(2,3)--(3,2)--(3,1)--cycle;

		\end{tikzpicture}
\end{tabular}

\\[.875 cm]
\hline
\end{tabular}
\caption{The different replacements for $U$ used in $\flip(U)$.
Sample updates are given with the convex hulls of $\uSet(z, U)$ and $\uSet(z, \overline{U})$ in dashed black lines and $\GP(\uSet)$ in red.
}\label{tableFlips}
\end{table}

We say that $\uSet$ is \emph{connected} if $\uSet \cap \GP(\uSet) \supseteq \{z, z+u^1\}$ (\textbf{Cases \hyperlink{case3}{3}} to \textbf{\hyperlink{case5}{5}}).
Otherwise, we say $\uSet$ is \emph{disconnected} (\textbf{Cases \hyperlink{case1}{1}} and \textbf{\hyperlink{case2}{2}}).
Note that $\uSet$ can be connected and not fit into \textbf{Cases \hyperlink{case3}{3}} to \textbf{\hyperlink{case5}{5}}; this occurs when $\uSet \cap \GP(\uSet)  =\uSet$.
On the other hand, if $\uSet$ is disconnected, then it must fit into the table.
Note that if \textbf{Case \hyperlink{case1}{1}} is executed, then the values $\sigma^1, \sigma^2$ defined in the table cannot both equal one by Lemma~\ref{Lemma:Gen_Gradient_Polyhedron}~\ref{Prop:partV}.
The importance of the connected case is that we are able to quickly determine if $\GP(\uSet)$ is lattice-free.
\begin{lemma}\label{lemma:lattice_freeness1}
Let $\uSet = \uSet(z,U)$ be connected and preprocessed as in Lemma~\ref{eqPreprocess}.
Then $\GP(\uSet)$ is lattice-free if and only if $H^{-1} \cap \Z^2 = \emptyset$ and $ H^1 \cap \Z^2 = \emptyset$.
\end{lemma}

\begin{proof}
	If $\GP(\uSet)$ is lattice-free, then $H^{-1} \cap \Z^2$ and $H^1 \cap \Z^2$ are empty.
	Assume to the contrary that $\GP(\uSet)$ is not lattice-free but $H^{-1} \cap \Z^2 = H^1 \cap \Z^2 = \emptyset$.
	Let ${x} \in \intr(\GP(\uSet)) \cap \Z^{2}$ be of any vector of the form ${x} = z + k_1u^1 + k_2u^2 $, where $k_1,k_2 \in \Z$ and such that $|k_2|$ is minimized.
	Note that $|k_2| \neq 1$ because $H^{-1} \cap \Z^2 = H^1 \cap \Z^2 = \emptyset$.
	If $k_2 = 0$ and $k_1 < 0$, then $z$ is a convex combination of $x$ and $z+u^1$ and is contained in $\intr(\GP(\uSet))$, which is a contradiction.
	Similarly, if $k_2 = 0$ and $k_1 \ge 1$, then $z+u^1 \in \intr(\GP(\uSet))$.
	Hence, $|k_2| \ge 2$.
	
	Both $z$ and $z+u^1$ are in $\GP(\uSet)$ because $\uSet$ is connected.
	The set $\conv\{z, z+u^1, {x}\}$ is contained in $\GP(\uSet)$ because $\GP(\uSet)$ is convex.
	From the linearity of the determinant and the equation $|\det(U)| = 1$, it follows that
	\[
		|\det(z - {x}, (z + u^1)- {x})| =|\det(-k_1u^1-k_2u^2,-(k_1-1)u^1-k_2u^2)| = |k_2| \geq 2.
	\]
	This implies that $\conv\{z, z+u^1, {x}\} \setminus \{z, z+u^1, {x}\} $ contains an integer vector $\overline{x} := z+k'_1u^1 + k'_2u^2$ with $k'_1, k'_2 \in \Z$ (see, e.g.,~\cite[Page 291, Corollary (2.6)]{barv2002}).
There are no integer vectors in $\conv\{z, z+u^1\} \setminus \{z,z+u^1\}$ because $U$ is unimodular, so $0 < |k'_2|<|k_2|$.
	Hence, $\overline{x} \in \intr(\GP(\uSet))$, which contradicts the choice of $k_2$. \qed
\end{proof}

Lemma~\ref{Lemma:Gen_Gradient_Polyhedron}~\ref{Prop:partII} implies $1 \le |\uSet \cap \GP(\uSet)| \le 4$.
However, Table~\ref{tableFlips} does not consider $|\uSet \cap \GP(\uSet)| = 4$ because $\GP(\uSet)$ is lattice-free in this setting.
\begin{corollary}\label{Cor:4_points_lattice_free}
		If $|\uSet \cap \GP(\uSet)| = 4$, then $\GP(\uSet)$ is lattice-free.
\end{corollary}

\begin{proof}
	If $|\uSet \cap \GP(\uSet)| = 4$, then $\uSet$ is connected.
Assume to the contrary that $z+ku^1+u^2 \in H^{1} \cap \Z^2$ for $k \ge 1$.
Then $z+u^1+u^2 \in \conv\{z+u^2, z+ku^1+u^2\}$, and Observation~\ref{Observation:interior} implies that $z+u^1+u^2 \in \intr(\GP(\uSet))$, which is a contradiction.
Using similar arguments, it can be shown that $H^{1} \cap \Z^2 = H^{-1} \cap \Z^2 = \emptyset$.
Hence, $\GP(\uSet)$ is lattice-free by Lemma~\ref{lemma:lattice_freeness1}.
\qed
\end{proof}

The previous results show that the cases in Table~\ref{tableFlips} suffice.

\begin{lemma}\label{lemNoCase}
Let $\uSet $ be preprocessed as in Lemma~\ref{eqPreprocess}.
If $\uSet$ does not fit into a case of Table~\ref{tableFlips}, then $\GP(\uSet)$ is lattice-free.
\end{lemma}
\begin{proof}
If $\uSet$ is connected, then the conditions of Lemma~\ref{lemma:lattice_freeness1} or Corollary~\ref{Cor:4_points_lattice_free} are met and $\GP(\uSet)$ is lattice-free.
	Otherwise, $\uSet$ is disconnected and falls into a case of the table.
	\qed
\end{proof}

In the remainder of this section we consider the set $\flip(\uSet)$ for every case in the table.
For this analysis, we briefly elaborate on the two measures of progress to prove Theorem~\ref{thmMain1} \emph{(ii)} and \emph{(iii)}.
The first measure is the smallest function value in $\uSet$.
Given another unimodular set ${\uSet}'$, we say
\begin{equation}\label{eqfNotation}
\uSet \lf \uSet' \text{ if } \min\{f(z) : z \in \uSet\} < \min\{f(z) : z \in \uSet'\}.
\end{equation}
We define $\uSet \lef \uSet'$ and $\uSet\ef \uSet'$ similarly.
Lemma~\ref{lemma:subset} establishes the inequality $\flip(\uSet) \lef \uSet$ for every case.
The proof follows directly from the definitions in the table and Lemma~\ref{propConvexProperties}.
\begin{lemma}\label{lemma:subset}
Let $\uSet $ be preprocessed as in Lemma~\ref{eqPreprocess}.
Then $\uSet \cap \GP(\uSet) \subseteq \flip(\uSet)$ and $\flip(\uSet) \leq_f \uSet$.
\end{lemma}

Our second measure of progress is~\eqref{eqDistVal}.
To motivate this, define the \emph{orthants} corresponding to $\uSet = \uSet(z,U)$:
\begin{equation}\label{eqQuadrants}
\begin{array}{rcll}
O_\uSet(z) & := & \{z + U r : r \in \R_{\le 0} \times \R_{\le 0}\},\\[.1 cm]
O_\uSet(z+u^1) & := &\{(z+u^1) + U r : r \in \R_{\ge 0} \times \R_{\le 0}\},\\[.1 cm]
O_\uSet(z+u^2) & := &\{(z+u^2) +U r: r \in \R_{\le 0} \times \R_{\ge 0}\}, \text{ and}\\[.1 cm]
O_\uSet(z+u^1+u^2) & := &\{(z+u^1+u^2) + Ur : r \in \R_{\ge 0} \times \R_{\ge 0}\}.
\end{array}
\end{equation}
See Figure~\ref{figCase1} \emph{(i)} for an example.
The matrix $U$ is unimodular, so for every $x \in \Z^2$ and $w \in \uSet$ the difference vector $x-w$ is an integer combination of $\|U^{-1}(x-w)\|_1$ many signed copies of $u^1$ and $u^2$.
Moreover, there is a unique $w^* \in \uSet$ such that $x \in O_{\uSet}(w^*)$, and this $w^*$ minimizes $\|U^{-1}(x-w)\|_1$ over $w \in \uSet$.
Denote this minimum value by
\[
r_\uSet( x) := \min\{\|U^{-1}(x-w)\|_1 : w \in \uSet\}.
\]
The value in~\eqref{eqDistVal} equals $\min\{r_\uSet(z^*) : z^* \text{ optimal for}~\eqref{eqMainProb}\}$.
Note $r_\uSet(x) = 0$ if and only if $x \in \uSet$.
%
%
If
\begin{equation}\label{eqfNotation2}
\begin{array}{rclcl}
&&\min\{ & r_\uSet(z^*) &: z^* \text{ is optimal for }\eqref{eqMainProb}\} \\
&< & \min\{& r_{\uSet'}(z^*) &: z^* \text{ is optimal for }\eqref{eqMainProb}\},
\end{array}
\end{equation}
then we write $\uSet\lr \uSet'$.
We define $\uSet\ler \uSet'$ and $\uSet \er \uSet'$ similarly.
The minima in~\eqref{eqfNotation2} exist due to Assumption~\eqref{eqLevelSets}.

The next result follows from the definition of $\er$ and Lemmata~\ref{Lemma:Gen_Gradient_Polyhedron}\ref{Prop:partI} and~\ref{lemma:subset}.
\begin{lemma}\label{lemContainMin}
If $\uSet$ contains a minimizer $z^*$ of~\eqref{eqMainProb}, then $z^* \in \flip(\uSet)$ and $r_{\uSet}(z^*) = r_{\flip(\uSet)}(z^*) = 0$.
\end{lemma}

We now proceed to analyze the cases independently.
In each subsection, we let $z^*$ denote an arbitrary optimal solution to~\eqref{eqMainProb}.

\subsection{An analysis of {\bf Case 1}.}

	\begin{lemma}\label{lemPart1ReduceROutcome}
		Suppose $\uSet$ fits into {\bf Case~\hyperlink{case1}{1}}.
	Then
	\begin{enumerate}[leftmargin = .75 cm, label = \textit{(\roman*)}]
	\item $\flip(\uSet) \ler \uSet$.
	\item if $\flip(\uSet) \ef \uSet$ and $\flip(\uSet) \er \uSet$, then $\flip(\uSet)$ is connected.
	\end{enumerate}
	\end{lemma}
	\begin{proof}
	First, we consider the outcome that
	\[
	\flip(U) = (-u^1, -u^2) ~~\text{and}~~\flip(\uSet) = \{z, z-u^1, z-u^2, z-u^1-u^2\}.
	\]	
See Figure~\ref{figCase1}.
	 Here, $z$ strictly cuts $z+u^1$ and $z+u^2$.
	By Lemma~\ref{Lemma:Gen_Gradient_Polyhedron}~\ref{Prop:partIV} we have
	\begin{equation}\label{EqCase1zcutsall}
	z ~\text{strictly cuts}~z+k_1u^1+k_2u^2 \quad \forall ~ k_1, k_2 \ge 0 ~\text{with}~k_1+k_2 \ge 1.
	\end{equation}
	This implies $z^* \not \in O_\uSet(z+u^1+u^2)$.
	Hence, $z^*\in O_\uSet(z) \cup O_\uSet(z+u^1)\cup O_\uSet(z+u^2)$.
	By~\eqref{EqCase1zcutsall}, if $z^* \in O_\uSet(z+u^2)$, then
	\[
	z^* = (z+u^2)+k^1(-u^1)+k_2u^2 = (z-u^1) + (k_1-1)(-u^1)+(k_2+1)u^2,
	\]
	where $k_1 \geq 1$ and $k_2\geq 0$.
	 This last expression implies that $z^* \in O_{\flip(\uSet)}(z-u^1)$ and
	\(
	r_{\flip(\uSet)}(z^*) = k_1+k_2 = r_{\uSet}(z^*).
	\)
	So, $\flip(\uSet) \leq_r \uSet$.
	We derive $\flip(\uSet) \leq_r \uSet$ similarly if $z^* \in O_\uSet(z+u^1) \cup O_\uSet(z)$.

		Suppose further that $\flip(\uSet) \ef \uSet$ and $\flip(\uSet) \er \uSet$.
	We claim that $\flip(\uSet)$ is connected.
	Here, we have $z \in \GP(\flip(\uSet))$, otherwise some vector in $\flip(\uSet)$ strictly cuts $z$ and $\flip(\uSet) \lf \uSet$ by Proposition~\ref{propConvexProperties}.
	Furthermore, $z$ does not strictly cut any vector in $\flip(\uSet)$ by~\eqref{EqCase1zcutsall}.
	Assume to the contrary that $\flip(\uSet)$ is disconnected.
	Then, $\flip(\uSet) \cap \GP(\uSet) = \{z, z-u^1-u^2\}$ by Lemma~\ref{Lemma:Gen_Gradient_Polyhedron} \ref{Prop:partV}.
		Lemma~\ref{Lemma:Gen_Gradient_Polyhedron} \ref{Prop:partIII} states $z-u^1-u^2$ strictly cuts $z-u^1$ and $z-u^1$.
 	Thus, $z-u^1-u^2$ strictly cuts $z$ by Lemma~\ref{Lemma:Gen_Gradient_Polyhedron}  \ref{Prop:partIV}, which is a contradiction.

	The other outcomes in {\bf Case~\hyperlink{case1}{1}} are $\flip(U) = (-u^1, u^2)$ and $\flip(U) = (u^1, -u^2)$.
	These settings are symmetric, so suppose
	\[
	\flip(U) = (u^1, -u^2) ~~\text{and}~~\flip(\uSet) = \{z, z+u^1, z-u^2, z+u^1-u^2\}.
	\]	
	Here, $z$ strictly cuts $z-u^1$ and $z+u^2$. Thus, by Lemma~\ref{Lemma:Gen_Gradient_Polyhedron}~\ref{Prop:partIV} $z$ strictly cuts every vector in $O_{\uSet}(z+u^2)$ and $z^* \not \in O_{\uSet}(z+u^2)$.
	Since $z+u^1+u^2 \not \in \GP(\uSet)$, $z+u^1+u^2$ is strictly cut by at least one vector in $\uSet \setminus \{z+u^1+u^2\}$.
	If $z+u^1+u^2$ is strictly cut by $z+u^1$ or $z+u^2$, then Lemma~\ref{Lemma:Gen_Gradient_Polyhedron}~\ref{Prop:partIV} implies that every vector in $O_{\uSet}(z+u^1+u^2)$ is strictly cut and hence $z^* \not \in O_{\uSet}(z+u^1+u^2)$.
	Suppose that neither $z+u^1$ nor $z+u^2$ strictly cut $z+u^1+u^2$.
	This implies that $z+u^2$ does not strictly cut $z+u^1$.
	Since $z$ does not strictly cut $z+u^1$ by assumption, $z+u^1$ is strictly cut by $z+u^1+u^2$.
	Because  $z+u^1+u^2$ does not strictly cut $z$, Lemma~\ref{Lemma:Gen_Gradient_Polyhedron}~\ref{Prop:partIV} yields that $z+u^1+u^2$ strictly cuts every vector of the form $(z+u^1+u^1) + (-k_1)u^1+(-k_2)u^2$ with $k_1 \geq k_2 \geq 0$.
	Similarly, since $z$ strictly cuts $z+u^1+u^2$ but does not strictly cut $z+u^1$, Lemma~\ref{Lemma:Gen_Gradient_Polyhedron}~\ref{Prop:partIV} yields that $z$ strictly cuts every vector of the form $(z+u^1+u^1) + (-k_1)u^1+(-k_2)u^2$ with $0 \leq k_1 \leq k_2$.
	Overall every vector in $O_{\uSet}(z+u^1+u^2)$ is strictly cut and thus $z^* \not \in O_{\uSet}(z+u^1+u^2)$.

	Hence, $z^* \in O_\uSet(z) \cup O_\uSet(z+u^1)$.
	If $z^* \in O_\uSet(z)$, then
	\[
	z^* = z +k_1 (-u^1) +k_2(-u^2),
	\]
	where $k_1, k_2 \ge 0$.
	If $k_2 = 0$, then $z^* \in O_{\flip(\uSet)}(z)$ and $r_{\flip(\uSet)}(z^*)  = k_1 = r_{\uSet}(z^*)$.
	If $k_2 \ge 1$, then $z^* \in O_{\flip(\uSet)}(z-u^2)$ and
	\[
	z^* = (z-u^2) +k_1 (-u^1) +(k_2-1)(-u^2),
	\]
	which shows that $r_{\uSet}(z^*)  = k_1 +k_2 > k_1+(k_2-1) = r_{\flip(\uSet)}(z^*)$.

	Symmetrically, if $z^* \in O_\uSet(z+u^1)$, then $r_{\flip(\uSet)}(z^*)  \ler r_{\uSet}(z^*)$.
	Furthermore, using a proof similar to above, we see that if $\flip(\uSet) \ef \uSet$ and $\flip(\uSet) \er \uSet$, then $\flip(\uSet)$ is connected.
	\qed
	\end{proof}

		\begin{figure}
	\hspace*{-.25 cm}
	\begin{tabular}{c @{\hskip 1 cm}c}
	\begin{tikzpicture}[scale = .5]

	\draw[red, opacity = .3, thick](0,5)--(5,0);
	\draw[red, opacity = .3, thick](1,5)--(6,0);

	\node[] at (-2.5,1) {$O_\uSet(z)$};
	\draw[draw = none, fill = black!30, opacity = .5](-1,0)--(2,0)--(2,2)--(-1,2)--cycle;

	\draw[draw = none, fill = black!30, opacity = .5](3,0)--(3,2)--(6,2)--(6,0)--cycle;
	\node[black] at (5,-0.5){$O_\uSet(z+u^1)$};

	\draw[draw = none, fill = black!30, opacity = .5](-1,3)--(2,3)--(2,5)--(-1,5)--cycle;
	\node[black] at (-3,4){$O_\uSet(z+u^2)$};

	\node[] at (4,5.5){$O_{\uSet}(z+u^1+u^2)$};
	\draw[draw = none, fill = black!30, opacity = .5](3,3)--(3,5)--(6,5)--(6,3)--cycle;

	\draw[dashed, thick, color = black](2,2)--(3,2)--(3,3)--(2,3)--cycle;


	\foreach \i in {-1,...,6}{
	\foreach \j in {0,...,5}
	\draw[fill = black!50, draw = black!50](\i,\j) circle (.5 ex);
	}

	 \draw[ draw = none, fill = red!30, opacity = .3](0,5)--(-1,5)--(-1,0)--(3.33333,0)--cycle;
	\draw[red, thick](0,5)--(3.333333,0);

	\draw[fill = black, draw = black](2,2) circle (.75 ex) node[black,  below left = -1pt]{$z$};
	\draw[fill = black, draw = black](3,2) circle (.75 ex) node[black, below right = -4pt]{$z+u^1$};
	\draw[fill = black, draw = black](2,3) circle (.75 ex) node[black, above left = -3pt]{$z+u^2$};
	\draw[fill = black, draw = black](3,3) circle (.75 ex);

	
	\end{tikzpicture}
	&
		\begin{tikzpicture}[scale = .5]

	\draw[draw = none, fill = black!30, opacity = .5](-1,0)--(2,0)--(2,2)--(-1,2)--cycle;
	\node[black] at (-.5,-0.5){$O_{\flip(\uSet)}(z-u^1-u^2)$};

	\draw[draw = none, fill = black!30, opacity = .5](3,0)--(3,2)--(6,2)--(6,0)--cycle;

	\draw[draw = none, fill = black!30, opacity = .5](-1,3)--(2,3)--(2,5)--(-1,5)--cycle;

	\node[] at (4.5,5.5){$O_{\flip(\uSet)}(z)$};
	\draw[draw = none, fill = black!30, opacity = .5](3,3)--(3,5)--(6,5)--(6,3)--cycle;

	\draw[dashed, thick, color = black](2,2)--(3,2)--(3,3)--(2,3)--cycle;


	\foreach \i in {-1,...,6}{
	\foreach \j in {0,...,5}
	\draw[fill = black!50, draw = black!50](\i,\j) circle (.5 ex);
	}

	\draw[fill = black, draw = black](2,2) circle (.75 ex);
	\draw[fill = black, draw = black](3,2) circle (.75 ex) node[black, below right = -4pt]{$z-u^2$};
	\draw[fill = black, draw = black](2,3) circle (.75 ex) node[black, above left = -3pt]{$z-u^1$};
	\draw[fill = black, draw = black](3,3) circle (.75 ex) node[black, above right = -1pt]{$z$};
	
	\end{tikzpicture}
	\\
	\hspace{1.5 cm}\emph{(i)} & \emph{(ii)} \hspace{1.5 cm}
	\end{tabular}
	\caption{\emph{(i)} \textbf{Case \protect \hyperlink{case1}{1}} when $\flip(U) = (-u^1,-u^2)$.
	$\GP(\uSet)$ is also drawn.
	\emph{(ii)} A sample update $\flip(U) = (-u^1, -u^2)$.
	}\label{figCase1}
	\end{figure}
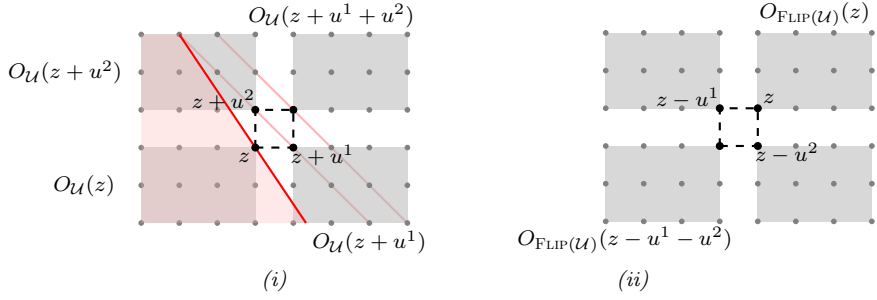

A defining property of {\bf Case~\hyperlink{case1}{1}} is that $z \in \uSet \cap \flip(\uSet)$.
From this, we make the following observations for later use.

\begin{lemma}\label{lemCase1Extras}
If $\uSet$ fits into {\bf Case~\hyperlink{case1}{1}}, then $\min\{f(v) : v\in \uSet\} = f(z)$.
Also, either $\flip(\uSet)$ is connected or $z$ is strictly cut by a vector in $\flip(\uSet)$.
\end{lemma}
\begin{proof}
The fact that $z$ minimizes $f$ over $\uSet$ follows from Proposition~\ref{propConvexProperties} because each vector in $\uSet \setminus\{z\}$ is strictly cut by some vector in $\uSet$.
If $z$ is not strictly cut by a vector in $\flip(\uSet)$, then $z \in \GP(\flip(\uSet))$ and $| \flip(\uSet) \cap \GP(\flip(\uSet)) | \ge 2$ by Lemma~\ref{Lemma:Gen_Gradient_Polyhedron} \ref{Prop:partV}.
Assume to the contrary that $\flip(\uSet)$ is disconnected.
Thus, $\flip(\uSet) = \{z,z+ \sigma^1u^1+\sigma^2u^2\}$, where $\sigma^1, \sigma^2$ are defined in Table \ref{tableFlips}.
The vector $z+ \sigma^1u^1+\sigma^2u^2$ must strictly cut $z+\sigma^1u^1, z+\sigma^2u^2 \in \flip(\uSet)$.
Then Lemma~\ref{Lemma:Gen_Gradient_Polyhedron} \ref{Prop:partIV} implies that $z+ \sigma^1u^1+\sigma^2u^2$ strictly cuts $z$, which is a contradiction.
\qed
\end{proof}

\subsection{An analysis of {\bf Case 2}.}
The following sets are helpful to analyze \textbf{Case~\protect\hyperlink{case2}{2}}:
	\[
	\begin{array}{rcl}
	O_\uSet(z+u^1+u^2)_B &:=& \{(z+u^1+u^2) + k_1(u^1+u^2) + k_2u^2: k_1,k_2 \in \R_{\ge 0}\},\\[.05 cm]
	O_\uSet(z+u^1+u^2)_A &:=& O_\uSet(z+u^1+u^2)\setminus O_\uSet(z+u^1+u^2)_B,\\[.05 cm]
	O_\uSet(z)_B &:=& \{z+k_1(u^1+u^2)+k_2u^2: k_1,k_2 \in \R_{ \le  0}\}, \text{ and}\\[.05 cm]
	O_\uSet(z)_A &:=& O_\uSet(z) \setminus O_\uSet(z)_B.
	\end{array}
	\]
	Figure~\ref{figCase2} \emph{(i)} illustrates these sets.

	\begin{figure}
	\hspace*{-.35 cm}
	\begin{tabular}{c@{\hskip .375 cm} c}
	\begin{tikzpicture}[scale = .5]

	\node[] at (-2.25,1){$O_\uSet(z)_A$};
	\node[] at (1,-0.5){$O_\uSet(z)_B$};
	\draw[draw = none, fill = black!30, opacity = .5](1,2)--(-1,2)--(-1,0)--cycle;
	\draw[draw = none, fill = black!30, opacity = .5](0,0)--(2,2)--(2,0)--cycle;

	\draw[draw = none, fill = black!30, opacity = .5](3,0)--(3,2)--(6,2)--(6,0)--cycle;
	\node[black] at (5,-0.5){$O_\uSet(z+u^1)$};

	\draw[draw = none, fill = black!30, opacity = .5](-1,3)--(2,3)--(2,5)--(-1,5)--cycle;
	\node[black] at (-1,5.5){$O_\uSet(z+u^2)$};

	\node[] at (8.8,4){$O_\uSet(z+u^1+u^2)_A$};
	\node[] at (4.5,5.5){$O_\uSet(z+u^1+u^2)_B$};
	\draw[draw = none, fill = black!30, opacity = .5](4,3)--(6,3)--(6,5)--cycle;
	\draw[draw = none, fill = black!30, opacity = .5](3,3)--(5,5)--(3,5)--cycle;

	\draw[dashed, thick, color = black](2,2)--(3,2)--(3,3)--(2,3)--cycle;


	\foreach \i in {-1,...,6}{
	\foreach \j in {0,...,5}
	\draw[fill = black!50, draw = black!50](\i,\j) circle (.5 ex);
	}

	\draw[fill = black, draw = black](2,2) circle (.75 ex) node[black,  below left = -1pt]{$z$};
	\draw[fill = black, draw = black](3,2) circle (.75 ex) node[black, below right = -4pt]{$z+u^1$};
	\draw[fill = black, draw = black](2,3) circle (.75 ex) node[black, above left = -3pt]{$z+u^2$};
	\draw[fill = black, draw = black](3,3) circle (.75 ex);

	\end{tikzpicture}
	&
	\begin{tikzpicture}[scale = .5]

		\foreach \i in {-1,...,6}{
	\foreach \j in {0,...,5}
	\draw[fill = black!50, draw = black!50](\i,\j) circle (.5 ex);
	}

		\node[black] at (3,-0.5){$O_{\flip(\uSet)}(z)$};
	\clip (-1.1,-.1) rectangle (6.1,5.1);
	
	\draw[draw = none, fill = black!30, opacity = .5](2,2)--(6,2)--(6,0)--(-2,0)--cycle;
	
	\draw[draw = none, fill = black!30, opacity = .5](1,2)--(-2,2)--(-1,1)--cycle;
	
	\draw[draw = none, fill = black!30, opacity = .5](4,3)--(6,3)--(6,4)--cycle;
	
	\draw[draw = none, fill = black!30, opacity = .5](3,3)--(-1,3)--(-1,5)--(7,5)--cycle;

	\draw[fill = black, draw = black](2,2) circle (.75 ex) node[black,  below right = -1pt]{$z$};
	\draw[fill = black, draw = black](1,2) circle (.75 ex) node[black,   above left = -3pt]{$z-u^1$};
	\draw[fill = black, draw = black](3,3) circle (.75 ex) node[black,  above left = -2pt]{$z+u^1+u^2$};
	
		\draw[fill = black, draw = black](4,3) circle (.75 ex);

	\draw[dashed, thick, draw=black](2,2)--(1,2)--(3,3)--(4,3)--cycle;

	\end{tikzpicture}
	\\
	\emph{(i)} \hspace*{1.25 cm} & \hspace*{1.25 cm} \emph{(ii)}
	\end{tabular}
	\caption{\emph{(i)} Sets used to analyze \textbf{Case~\protect\hyperlink{case2}{2}}.
	%
	%
	\emph{(ii)} A sample update for the third outcome of \textbf{Case~\protect\hyperlink{case2}{2}}, $\flip(U) = (-u^1, 2u^1+u^2)$.
	}\label{figCase2}
	\end{figure}
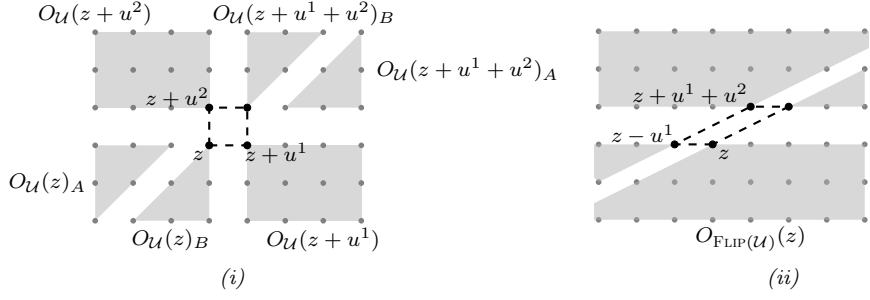

		\begin{lemma}\label{lemPart2ReduceROutcome}
		Suppose $\uSet$ fits into {\bf Case~\hyperlink{case2}{2}}.
	Then
	\begin{enumerate}[leftmargin = .75 cm, label = \textit{(\roman*)}]
	\item\label{lemCase2i} $\flip(\uSet) \ler \uSet$.
	\item\label{lemCase2ii} if $\flip(U) = (-u^1, 2u^1+u^2)$ and $\uSet$ does not contain a minimizer of~\eqref{eqMainProb}, then $\flip(\uSet) \lr \uSet$.
	\item\label{lemCase2iii} if $\flip(\uSet) \ef \uSet$ and $\flip(\uSet) \er \uSet$, then $\flip(\uSet)$ is connected.
	\end{enumerate}
	\end{lemma}
	\begin{proof}
		First, we prove $\flip(\uSet) \leq_r \uSet$.
	By Lemma~\ref{lemContainMin} we assume that $ \uSet$ does not contain a minimizer $z^*$ of~\eqref{eqMainProb}.
	By Lemma~\ref{Lemma:Gen_Gradient_Polyhedron} \ref{Prop:partI}, $z^* \in \intr(\GP(\uSet))$.
	If $z^* \in O_\uSet(z+u^2)$, then $z+u^2 \in \conv\{z^*, z+u^1+u^2, z\}$ giving the contradiction $z+u^2\in \intr(\GP(\uSet))$.
	By symmetry, $z^* \not\in O_\uSet(z+u^1)$.
		 The preprocessing in Lemma~\ref{eqPreprocess} states that $z+u^1+u^2$ strictly cuts $z+u^2$ and does not cut $z+u^1$.
		 Thus, $z+u^1+u^2$ strictly cuts every vector in $O_\uSet(z+u^1+u^2)_B $ by Lemma~\ref{Lemma:Gen_Gradient_Polyhedron} \ref{Prop:partIV}.
		So, $z^* \not \in O_\uSet(z+u^1+u^2)_B$.
	Similarly, $z^* \not\in O_\uSet(z)_B$.
	Therefore, $z^* \in O_\uSet(z)_A \cup O_\uSet(z+u^1+u^2)_A$.
	As implied by Figure~\ref{figCase2} \emph{(i)}, the sets $O_\uSet(z)_A$ and $O_\uSet(z+u^1+u^2)_A$ are symmetric around $\uSet$.
	Hence, we simplify the proof of $\flip(\uSet) \ler \uSet$ by assuming without loss of generality that $z^* \in  O_\uSet(z)_A$.

	\textbf{{Case \hyperlink{case2}{2}}} has three outcomes.
	The first two update $U$ to one of the matrices $U'$ and $U''$ from Table \ref{tableFlips}.
	The proofs of correctness for each of these outcomes are similar.
	We prove $\flip(\uSet) \leq_r \uSet$ when
	\[
	\begin{array}{rl}
	& \flip(U) = U' = (u^1, u^1+u^2) \\[.1 cm]
	\text{and}& \flip(\uSet) = \uSet' = \{z, z+u^1, z+u^1+u^2, z+2u^1+u^2\}.
	\end{array}
	\]
See Figure~\ref{figCase22}.
	\begin{figure}
	\hspace*{-.25 cm}
	\begin{tabular}{c @{\hskip 1.75 cm}c}
	\begin{tikzpicture}[scale = .5]

	\draw[red, opacity = .3, thick] (0,0)--(6,4);
	\draw[red, opacity = .3, thick](-1,2)--(6,4.33333);

	\node[] at (-2.5,1) {$O_\uSet(z)$};
	\draw[draw = none, fill = black!30, opacity = .5](-1,0)--(2,0)--(2,2)--(-1,2)--cycle;

	\draw[draw = none, fill = black!30, opacity = .5](3,0)--(3,2)--(6,2)--(6,0)--cycle;
	\node[black] at (5,-0.5){$O_\uSet(z+u^1)$};

	\draw[draw = none, fill = black!30, opacity = .5](-1,3)--(2,3)--(2,5)--(-1,5)--cycle;
	\node[black] at (-3,4){$O_\uSet(z+u^2)$};

	\node[] at (5,5.5){$O_{\uSet}(z+u^1+u^2)$};
	\draw[draw = none, fill = black!30, opacity = .5](3,3)--(3,5)--(6,5)--(6,3)--cycle;

	\draw[dashed, thick, color = black](2,2)--(3,2)--(3,3)--(2,3)--cycle;


	\foreach \i in {-1,...,6}{
	\foreach \j in {0,...,5}
	\draw[fill = black!50, draw = black!50](\i,\j) circle (.5 ex);
	}

	\draw[ draw = none, fill = red!30, opacity = .3](-1,2)--(-1,0.5)--(5,3.5)--cycle;
	\draw[red, thick](-1,0.5)--(6,4);
	\draw[red, thick](-1,2)--(6,3.75);

	\draw[fill = black, draw = black](2,2) circle (.75 ex) node[black,  above left = -1pt]{$z$};
	\draw[fill = black, draw = black](3,2) circle (.75 ex) node[black, below right = -4pt]{$z+u^1$};
	\draw[fill = black, draw = black](2,3) circle (.75 ex) node[black, above left = -3pt]{$z+u^2$};
	\draw[fill = black, draw = black](3,3) circle (.75 ex);

	
	
	\end{tikzpicture}
	&
		\begin{tikzpicture}[scale = .5]

	\draw[red, opacity = .3, thick] (0,0)--(6,4);

	\draw[draw = none, fill = black!30, opacity = .5](3,2)--(6,2)--(6,0)--(1,0)--cycle;

	\draw[draw = none, fill = black!30, opacity = .5] (3,3)--(5,5)--(-1,5)--(-1,3)--cycle;

	\draw[draw = none, fill = black!30, opacity = .5](2,2) -- (0,0)--(-1,0)--(-1,2)--cycle;

	\draw[draw = none, fill = black!30, opacity = .5] (4,3)--(6,5)--(6,5)--(6,3)--cycle;

	\draw[dashed, thick, color = black](2,2)--(3,2)--(4,3)--(3,3)--cycle;
	
	\node[black] at (5,-0.5){$O_{\flip(\uSet)}(z+u^1)$};


	\foreach \i in {-1,...,6}{
	\foreach \j in {0,...,5}
	\draw[fill = black!50, draw = black!50](\i,\j) circle (.5 ex);
	}


		\draw[ draw = none, fill = red!30, opacity = .3](-1,2)--(-1,0.5)--(4,3)--(3,3)--cycle;
	\draw[red, thick](-1,0.5)--(6,4);
	\draw[red, thick](-1,2)--(6,3.75);
	\draw[red, thick](-1,3)--(6,3);

	\draw[fill = black, draw = black](2,2) circle (.75 ex) node[black,  above left = -1pt]{$z$};
	\draw[fill = black, draw = black](3,2) circle (.75 ex) node[black, below right = -3pt]{$z+u^1$};
	 \draw[fill = black, draw = black](3,3) circle (.75 ex)  node[black, above left = -1pt]{$z+u^1+u^2$};
	 \draw[fill = black, draw = black](4,3) circle (.75 ex);

	 \draw[fill = black, draw = black](-1,1) circle (.75 ex) node[black,  above right = -1pt]{$z^*$};

	
	\end{tikzpicture}
	\\
	\hspace{1.5 cm}\emph{(i)} & \emph{(ii)} \hspace{1.5 cm}
	\end{tabular}
	\caption{\emph{(i)} \textbf{Case \protect \hyperlink{case2}{2}} where $\flip(\uSet) = \uSet'$.
	$\GP(\uSet)$ is also drawn.
	\emph{(ii)} A sample update $\flip(\uSet) = \uSet'$. $\GP(\uSet')$ is also drawn. Note that $\GP(\uSet')$ is connected.
	}\label{figCase22}
	\end{figure}
 Here we have
	\[
	z^*=z+k_1(-u^1)+k_2(-u^2) = z+(k_1-k_2)u^1+k_2(-u^1-u^2),
	\]
	 where $k_1,k_2 \geq 0$ and $k_1\geq k_2+1$.
	Thus,
	\[
	r_\uSet(z^*) = k_1+k_2 \geq (k_1-k_2)+k_2 = r_{\flip(\uSet)}(z^*),
	\]
	and $\flip(\uSet) \ler \uSet$.

		Next, we consider the third outcome of \textbf{{Case \hyperlink{case2}{2}}}:
	\[
	\flip(U) = (-u^1, 2u^1+u^2) ~~\text{and}~~ \flip(\uSet) = \{z, z-u^1, z+u^1+u^2, z+2u^1+u^2\}.
	\]
	In this setting $z^* \in O_{\flip(\uSet)}(z-u^1) \cup O_{\flip(\uSet)}(z)$.
	This is indicated in Figure~\ref{figCase2}.
	The proof when $z^* \in O_{\flip(\uSet)}(z)$ is similar to when $z^* \in O_{\flip(\uSet)}(z- u^1)$, so we consider $z^* \in O_{\flip(\uSet)}(z)$.
	There exist $k_1, k_2 \in \Z_{\ge 0}$ such that
	\begin{align*}
		z^* = z+k_1 (-u^1) + k_2(-u^2) =  z+(-k_1 + 2k_2)u^1+k_2(-2u^1-u^2) .
	\end{align*}
	We have $k_1 > k_2$ because $z^* \in O_\uSet(z)_A$.
	This is equivalent to $k_2 > (-k_1+2k_2)$.
	Thus,
	\begin{align*}
		r_\uSet(z^*) = k_1 + k_2 > (-k_1+2k_2)+k_2 = r_{\flip(\uSet)}(z^*)
	\end{align*}
	and $\flip(\uSet) <_r \uSet$.
	This proves~\ref{lemCase2i} and~\ref{lemCase2ii}.
	
	Now, suppose that $\flip(\uSet) \ef \uSet$ and $\flip(\uSet) \er \uSet$ and $\flip(\uSet) = \uSet'$ or $\flip(\uSet) = \uSet''$.
	We give the proof when $\flip(\uSet) = \uSet'$ as the other proof is similar.
	The definition of $\flip(\uSet) = \uSet'$ in Table~\ref{tableFlips} implies that one of three things must occur: $\uSet' $ is connected, $\GP(\uSet') \cap \uSet' = \{{z}+2u^1+u^2\}$, or $z+u^1+u^2$ strictly cuts $z-u^1$ and $z-u^1$ strictly cuts $z$.
	If $\uSet' $ is connected, then~\ref{lemCase2iii} holds.
	If $\uSet' \cap \GP(\uSet') = \{z+2u^1+u^2\}$, then $\uSet' \lf \uSet$, which contradicts $\flip(\uSet) \ef \uSet$.

	It remains to consider when $z+u^1+u^2$ strictly cuts $z-u^1$ and $z-u^1$ strictly cuts $z$.
	Here, $f(z+u^1+u^2) < f(z)$ and $f(z+u^1+u^2) = \min\{f(w) : w \in \uSet\}$.
	Moreover, $f(z+u^1+u^2) =\min\{f(w):w \in \uSet'\}$ because $\flip(\uSet) \ef \uSet$.
	Thus, $z+u^1+u^2 \in \uSet' \cap \GP(\uSet')$ by Proposition~\ref{propConvexProperties}.
	The definition of \textbf{Case~\hyperlink{case2}{2}} and the preprocessing in Lemma~\ref{eqPreprocess} imply that $z+u^1+u^2$ does not strictly cut any vector in $\uSet' \setminus \{z+u^1+u^2\}$.
  By Lemma~\ref{Lemma:Gen_Gradient_Polyhedron} \ref{Prop:partV} we have $|\uSet ' \cap \GP(\uSet')| \ge 2$.
	%
	%
	If $\uSet'$ was disconnected, then $\uSet' \cap \GP(\uSet') = \{z+u^1+u^2, z+u^1\}$.
	However, $z +u^1 \not \in \uSet' \cap \GP(\uSet')$ because it is strictly cut by $z$ by the preprocessing on $\uSet$.
	Hence, $\uSet'$ is connected.
	In fact, we have shown the stronger statement that if $\flip(\uSet) \ef \uSet$, $\flip(\uSet) \er \uSet$, and $f(z+u^1+u^2) \le f(z)$, then $\uSet'$ is connected.

	It is left to prove~\ref{lemCase2iii} when $\flip(\uSet) \ef \uSet$, $\flip(\uSet) \er \uSet$, and $\flip(U) = (-u^1, 2u^1+u^2)$.
	We claim that this cannot occur.
	By~\ref{lemCase2ii} it must be the case that $\uSet \cap \GP(\uSet) = \{z,z+u^1+u^2\}$ contains a minimizer of~\eqref{eqMainProb}.
	The set $ \{z,z+u^1+u^2\}$ is contained in $\uSet' \cap \uSet''$.
	Hence, $\uSet\ef \uSet' \ef \uSet''$ and $\uSet\er \uSet' \er \uSet''$.
	We demonstrated in the previous paragraph that $\uSet'$ is connected if $f(z+u^1+u^2) \le f(z)$ and that $\uSet''$ is connected if $f(z) \le f(z+u^1+u^2)$.
	Hence, $\flip(U)$ will not equal $(-u^1, 2u^1+u^2)$ in this setting but rather $\flip(\uSet) = \uSet'$ or $\flip(\uSet) = \uSet''$.
	\qed
	\end{proof}
%

Unlike the connected cases, gradient information does not seem sufficient to make a precise flip in {\bf Case~\hyperlink{case2}{2}} without `guessing'.
This guessing is the cause of the more involved subcases.
The necessity of the three conditions for $\flip(\uSet) = \uSet'$ or $\flip(\uSet) = \uSet''$ is reflected in the next result.

\begin{lemma}\label{lemCase2Extras}
Assume $\uSet$ fits into {\bf Case~\hyperlink{case2}{2}} and let $w^*$ minimize $f(v)$ over $\uSet$.
If $\flip(\uSet) = \uSet'$ or $\flip(\uSet) = \uSet''$, then $\flip(\uSet)$ is connected or $w^*$ is strictly cut by a vector in $\flip(\uSet)$.
\end{lemma}
\begin{proof}
The outcomes $\flip(\uSet) = \uSet'$ or $\flip(\uSet) = \uSet''$ are symmetric, so we assume
	\[
	 \flip(\uSet) = \uSet' = \{z, z+u^1, z+u^1+u^2, z+2u^1+u^2\}.
	\]
The definition of $\flip(\uSet) = \uSet'$ in Table~\ref{tableFlips} implies that one of three things must occur: $\uSet' $ is connected, $\GP(\uSet') \cap \uSet' = \{{z}+2u^1+u^2\}$, or $z+u^1+u^2$ strictly cuts $z-u^1$ and $z-u^1$ strictly cuts $z$.
The first two conditions imply that the lemma holds.
So, assume that $z+u^1+u^2$ strictly cuts $z-u^1$ and $z-u^1$ strictly cuts $z$.

This sequence of strict cuts implies $w^* = z+u^1+u^2$.
If $z+u^1+u^2 \not\in \uSet'\cap \GP(\uSet')$, then the lemma holds.
So, assume $z+u^1+u^2 \in \uSet'\cap \GP(\uSet')$.
Then, $z+u^1+u^2$ is not strictly cut by any vector in $\uSet'$.
By the preprocessing of {\bf Case~\hyperlink{case2}{2}}, the vector $z+u^1+u^2$ strictly cuts $z+u^2$, so it does not cut $z+2u^1+u^2$.
Also, $z+u^1+u^2$ does not strictly cut $z$ or $z+u^1$.
Lemma~\ref{Lemma:Gen_Gradient_Polyhedron} \ref{Prop:partV} then implies $|\GP(\uSet') \cap \uSet'| \ge 2$.
One vector in $\GP(\uSet') \cap \uSet'$ is $z+u^1+u^2$.
The vector $z+u^1$ is strictly cut by $z$ because $\uSet$ falls into {\bf Case~\hyperlink{case2}{2}}, so the second vector in $\GP(\uSet') \cap \uSet'$ must be $z$ or $z+2u^1+u^2$.
Hence, $\uSet'$ is connected.
\qed
\end{proof}
	
\subsection{An analysis of {\bf Case 3}.}

\textbf{{Case \hyperlink{case3}{3}}} is executed if $|H^{-1} \cap \Z^2 | = 1$ or $|H^1 \cap \Z^2|=1$.
The two cases are symmetric, so we often assume $|H^1 \cap \Z^2| = 1$.
	See Figure~\ref{figCase3}.	
			\begin{figure}
	\hspace*{-.25 cm}
	\begin{tabular}{c @{\hskip 1.35 cm}c}
	\begin{tikzpicture}[scale = .5]

	\draw[red, opacity = .3, thick](-1,1)--(6,4.5);
	\draw[red, opacity = .3, thick](-1,1.5)--(6,5);

	\node[] at (-2.5,1) {$O_\uSet(z)$};
	\draw[draw = none, fill = black!30, opacity = .5](-1,0)--(2,0)--(2,2)--(-1,2)--cycle;

	\draw[draw = none, fill = black!30, opacity = .5](3,0)--(3,2)--(6,2)--(6,0)--cycle;
	\node[black] at (5,-0.5){$O_\uSet(z+u^1)$};

	\draw[draw = none, fill = black!30, opacity = .5](-1,3)--(2,3)--(2,5)--(-1,5)--cycle;
	\node[black] at (-.5,5.5){$O_\uSet(z+u^2)$};

	\node[] at (4.5,5.5){$O_\uSet(z+u^1+u^2)$};
	\draw[draw = none, fill = black!30, opacity = .5](3,3)--(3,5)--(6,5)--(6,3)--cycle;

	\draw[dashed, thick, color = black](2,2)--(3,2)--(3,3)--(2,3)--cycle;


	\foreach \i in {-1,...,6}{
	\foreach \j in {0,...,5}
	\draw[fill = black!50, draw = black!50](\i,\j) circle (.5 ex);
	}

		\draw[ draw = none, fill = red!30, opacity = .3](0,1)--(6,4)--(6,3)--cycle;
	\draw[red, thick](-1,.5)--(6,4);
	\draw[red, thick](-1,2/3)--(6,3);
		\draw[fill = red, draw = red](5,3) circle (.75 ex) node[above right, red]{$x$};

	\draw[fill = black, draw = black](2,2) circle (.75 ex) node[black,  above left = -1pt]{$z$};
	\draw[fill = black, draw = black](3,2) circle (.75 ex) node[black, below right = -4pt]{$z+u^1$};
	\draw[fill = black, draw = black](2,3) circle (.75 ex) node[black, above left = -3pt]{$z+u^2$};
	\draw[fill = black, draw = black](3,3) circle (.75 ex);

	\end{tikzpicture}
	&
	\begin{tikzpicture}[scale = .5]

	\node[] at (2,5.5) {$O_{\flip(\uSet)}(z)$};
	\node[] at (2,-0.5) {$O_{\flip(\uSet)}(z+u^1)$};
	\node[] at (8,3.5) {$O_{\flip(\uSet)}(x)$};
	
	\clip (-1.1,-.1) rectangle (6.1,5.1);
	\draw[draw = none, fill = black!30, opacity = .5](6,4)--(2,2)--(-1,1)--(-1,5)--(6,5)--cycle;

	\node[] at (2,-0.5) {$O_{\flip(\uSet)}(z)$};
	\draw[draw = none, fill = black!30, opacity = .5](6,3)--(3,2)--(-1,0)--(6,0)--cycle;
	
	\draw[draw = none, fill = black!30, opacity = .5](5,3)--(7,4)--(8,4)--cycle;
	
	\draw[draw = none, fill = black!30, opacity = .5](0,1)--(-2,0)--(-3,0)--cycle;


	\foreach \i in {-1,...,6}{
	\foreach \j in {0,...,5}
	\draw[fill = black!50, draw = black!50](\i,\j) circle (.5 ex);
	}

	\draw[fill = black, draw = black](2,2) circle (.75 ex) node[black,  above left = -1pt]{$z$};
	\draw[fill = black, draw = black](3,2) circle (.75 ex) node[black, below right  = -2 pt ]{$z+u^1$};
	\draw[fill = black, draw = black](0,1) circle (.75 ex) ;

	\draw[fill = black, draw = black](5,3) circle (.75 ex) node[above]{$x$};
	
	\draw[dashed, thick, color = black](2,2)--(0,1)--(3,2)--(5,3)--cycle;

	\end{tikzpicture}
	\\
	\hspace{1.5 cm}\emph{(i)} & \emph{(ii)} \hspace{1.5 cm}
	\end{tabular}
	\caption{\emph{(i)} The orthants used to prove Theorem~\ref{thmMain1Again} in \textbf{Case~\protect\hyperlink{case3}{3}}.
	$\GP(\uSet)$ is also drawn.
	%
	%
	\emph{(ii)} The orthants of $\flip(\uSet)$ when $|H^1 \cap \GP(\uSet)| = 1$.
	}\label{figCase3}
	\end{figure}

	\begin{lemma}\label{lemPart3ReduceROutcome}
		Suppose $\uSet$ fits into {\bf Case~\hyperlink{case3}{3}}.
	Then
	\begin{enumerate}[leftmargin = .7 cm, label = \textit{(\roman*)}]
	\item $\flip(\uSet) \ler \uSet$.
	\item if $\uSet$ does not contain a minimizer of~\eqref{eqMainProb}, then $\flip(\uSet) \lr \uSet$.
	\item if $\flip(\uSet) \ef \uSet$ and $\flip(\uSet) \er \uSet$, then $\flip(\uSet)$ is connected and $\flip(\uSet) \cap \GP(\flip(\uSet)) \setminus \uSet \neq \emptyset$.
	\end{enumerate}
	\end{lemma}
	\begin{proof}
	First we show that $\flip(\uSet) \leq_r \uSet$.
	By Lemma~\ref{lemContainMin} we assume that $ \uSet$ does not contain a minimizer $z^*$ of~\eqref{eqMainProb}.
	By Lemma~\ref{Lemma:Gen_Gradient_Polyhedron}\ref{Prop:partI}, $z^* \in \intr(\GP(\uSet))$.
	Using symmetry, we assume without loss of generality that $H^1 \cap \Z^2 = \{x\}$ and $x := z + ku^1+u^2$ for $k \ge 2$.
	The setting when $k \le -1$ is also symmetric.
	Here we have
	\[
	\begin{array}{l}
	\flip(U) = (ku^1+u^2, -(k-1)u^1 - u^2) ~~\text{and}\\[.1 cm]
	\flip(\uSet) = \{z, x, z+u^1, z-(k-1)u^1-u^2\}.
	\end{array}
	\]

	If $z^* \in O_\uSet(z+u^1)$, then $z+u^1 \in \conv\{z^*,z, x\}$ yielding the contradiction $z+u^1 \in \intr(\GP(\uSet))$.
	A similar argument shows $z^* \not \in O_\uSet(z+u^2)$.
	Hence, $z^* \in O_\uSet(z) \cup O_\uSet(z+u^1+u^2)$.
	The cases are symmetric, so we assume $z^* \in O_\uSet(z+u^1+u^2)$.
	Thus,
	\[
	z^* =(z+u^1+u^2) + k_1 u^1 + k_2 u^2,
	\]
	 where $k_1, k_2 \ge 0$ and $k_1+ k_2 \ge 1$.
	 If $z^* \in O_{\flip(\uSet)}(z) \cap O_\uSet(z+u^1+u^2)$, then $x-u^1 \in \conv\{z,x,z^*\}$  yielding the contradiction $x-u^1 \in \intr(\GP(\uSet))$.
	 Similarly, $z^* \not \in O_{\flip(\uSet)}(z+u^1) \cap O_\uSet(z+u^1+u^2)$ otherwise $x+u^1 \in \intr(\GP(\uSet))$.
	 Hence, $z^* \in [O_{\flip(\uSet)}(x) \cup O_{\flip(\uSet)}(z - (k-1)u^1- u^2)]\cap O_\uSet(z+u^1+u^2)$.
	 We provide the analysis when $z^* \in O_{\flip(\uSet)}(x) \cap O_\uSet(z+u^1+u^2)$ as the other case is similar.
Here, $z^*$ can be rewritten as
\[
x+(k_1+k_2+1-k-k_2k)(ku^1+u^2)+(k+k_2k-1-k_1)((k-1)u^1+u^2),
\]
and $k_1+k_2+1-k-k_2k$ and $k+k_2k-1-k_1$ are nonnegative.
Thus,
\[
r_\uSet(z^*) = k_1+k_2 >  k_2 = r_{\flip(\uSet)}(z^*)
\]
and $\flip(\uSet) <_r \uSet$.

Finally, suppose that $\flip(\uSet) \ef \uSet$.
Lemma~\ref{lemma:subset} and the equation $\flip(\uSet) \ef \uSet$ imply that $\{z, z+u^1\} \cap \flip(\uSet) \cap \GP(\flip(\uSet)) \neq \emptyset$.
	We show that $x$ or $z-(k-1)u^1-u^2$ is in $\GP(\flip(\uSet))$, which will prove \emph{(iii)}.
	By Proposition~\ref{propConvexProperties} it is enough to show that $z$ and $z+u^1$ do not cut any vector in $\flip(\uSet)$.
The definition of $\flip(\uSet)$ in Table~\ref{tableFlips} implies that $z$ and $z+u^1$ do not cut each other, and they do not cut $ x $.
	The vectors $x-u^1$ and $x+u^1$ are cut by vectors in $\uSet$ because \textbf{Case~\hyperlink{case3}{3}} is executed (as opposed to say \textbf{Case~\hyperlink{case4}{4}}).
	The vectors $z+u^2$ and $z+u^1+u^2$ do not cut any vectors of the form $z+ru^1+u^2$ for $r \ge 0$, otherwise they cut $x$ and \textbf{Case~\hyperlink{case3}{3}} would not be executed.
	Hence, $x-u^1$ and $x+u^1$ are cut by $z$ and $z+u^1$.
	If $z$ strictly cuts $x+u^1$ but not $x$, then it also strictly cuts $z+u^1$, contradiction.
	Thus, $z+u^1$ cuts $x+u^1$, which implies that $z+u^1$ does not cut $z-(k-1)u^1-u^2$.
	Similarly, $z$ cuts $x-u^1$ and does not cut $z-(k-1)u^1-u^2$.
	\qed
	\end{proof}
	
\subsection{An analysis of {\bf Case 4}.}

\textbf{{Case \hyperlink{case4}{4}}} is executed if $|H^{-1} \cap \Z^2 | \ge 2$ or $|H^1 \cap \Z^2| \ge 2$.
The two cases are symmetric, so we always assume $|H^1 \cap \Z^2| \ge 2$.
	Furthermore, we assume that the vector $x := z+ku^1 + u^2 \in H^1\cap \Z^2$ minimizing $|k|$ satisfies $k \ge 2$; if $k \le -1$, then similar arguments can be made.
	Here, we have
	\begin{equation}\label{eqAssumptions4}
	\flip(U) = (u^1, x) ~~\text{and} ~~\flip(\uSet) = \{z, x, z+u^1, x+u^1\}.
	\end{equation}
Figure~\ref{figCase4} illustrates this case.
	\begin{figure}
	\begin{tabular}{c @{\hskip .5 cm}c}
	\begin{tikzpicture}[scale = .5]

	\node[] at (0,-.5) {$O_\uSet(z)$};
	\node[black] at (5,-0.5){$O_\uSet(z+u^1)$};
	\node[black] at (1,5.5){$O_\uSet(z+u^2)$};
	\node[] at (6,5.5){$O_\uSet(z+u^1+u^2)$};
	
	\clip (-.1,-.1) rectangle (7.1,5.1);
	\draw[red, opacity = .3, thick](-1,1/3)--(9,7);
	\draw[red, opacity = .3, thick](-1,1)--(8,7);


	\draw[draw = none, fill = black!30, opacity = .5](-1,0)--(2,0)--(2,2)--(-1,2)--cycle;

	\draw[draw = none, fill = black!30, opacity = .5](3,0)--(3,2)--(7,2)--(7,0)--cycle;

	\draw[draw = none, fill = black!30, opacity = .5](-1,3)--(2,3)--(2,6)--(-1,6)--cycle;


	\draw[draw = none, fill = black!30, opacity = .5](3,3)--(3,6)--(7,6)--(7,3)--cycle;

	\draw[dashed, thick, color = black](2,2)--(3,2)--(3,3)--(2,3)--cycle;


	\foreach \i in {-1,...,7}{
	\foreach \j in {0,...,6}
	\draw[fill = black!50, draw = black!50](\i,\j) circle (.5 ex);
	}

		\draw[ draw = none, fill = red!30, opacity = .3](1,1+1/3)--(9,6+2/3)--(9,4)--cycle;
	\draw[red, thick](-1,0)--(9,6+2/3);
	\draw[red, thick](-1,2/3)--(9,4);
	
	\draw[fill = red, draw = red](4,3) circle (.75 ex) node[right, red]{$x$};
	
	\draw[fill = black, draw = black](2,2) circle (.75 ex) node[black,  above left = -1pt]{$z$};
	\draw[fill = black, draw = black](3,2) circle (.75 ex) node[black, below right = -4pt]{$z+u^1$};
	\draw[fill = black, draw = black](2,3) circle (.75 ex) node[black, above left = -3pt]{$z+u^2$};
	\draw[fill = black, draw = black](3,3) circle (.75 ex);

	\end{tikzpicture}
	&

	\begin{tikzpicture}[scale = .5]

	\node[] at (-1.5,1.5) {$O_{\flip(\uSet)}(z)$};
	\node[] at (3,-0.5){$O_{\flip(\uSet)}(z+u^1)$};
	\node[] at (3,5.5){$O_{\flip(\uSet)}(z+u^1+u^2)$};
		
	\clip (-.1,-.1) rectangle (7.1,5.1);


	\draw[draw = none, fill = black!30, opacity = .5](0,1)--(0,2)--(2,2)--cycle;

	\draw[draw = none, fill = black!30, opacity = .5](0,0)--(0,.5)--(3,2)--(7,2)--(7,0)--cycle;

	\draw[draw = none, fill = black!30, opacity = .5](4,3)--(0,3)--(0,6)--(7,6)--(7, 4.5)--cycle;

	\draw[draw = none, fill = black!30, opacity = .5](5,3)--(7,3)--(7,4)--cycle;

	\draw[dashed, thick, color = black](2,2)--(3,2)--(5,3)--(4,3)--cycle;


	\foreach \i in {-1,...,7}{
	\foreach \j in {0,...,6}
	\draw[fill = black!50, draw = black!50](\i,\j) circle (.5 ex);
	}

	\draw[fill = black, draw = black](2,2) circle (.75 ex) node[black,  above left = -1pt]{$z$};
	\draw[fill = black, draw = black](3,2) circle (.75 ex) node[black, below right = -4pt]{$z+u^1$};
	\draw[fill = black, draw = black](5,3) circle (.75 ex); 

	\draw[fill = black, draw = black](4,3) circle (.75 ex) node[above left]{$x$};
	
	\end{tikzpicture}

	\\
	\hspace{1.5 cm}\emph{(i)} & \emph{(ii)}
	\end{tabular}
	\caption{\emph{(i)} The orthants used to prove Theorem~\ref{thmMain1Again} in \textbf{Case~\protect\hyperlink{case4}{4}}.
	$\GP(\uSet)$ is also drawn.
	%
	%
	\emph{(ii)} The orthants of $\flip(\uSet)$ when $|H^1 \cap \GP(\uSet)| \ge 2$.
	}\label{figCase4}
	\end{figure}

	\begin{lemma}\label{lemPart4ReduceROutcome}
	Suppose $\uSet$ fits into {\bf Case~\hyperlink{case4}{4}}.
	Then
	\begin{enumerate}[leftmargin = .75 cm, label = \textit{(\roman*)}]
	\item $\flip(\uSet) \ler \uSet$.
	\item if $\uSet$ does not contain a minimizer of~\eqref{eqMainProb}, then $\flip(\uSet) \lr \uSet$.
	\item if $\flip(\uSet) \ef \uSet$ and $\flip(\uSet) \er \uSet$, then $\flip(\uSet)$ is connected and $\flip(\uSet) \cap \GP(\flip(\uSet)) \setminus \uSet \neq \emptyset$.
	\end{enumerate}
	\end{lemma}
\begin{proof}
By symmetry we assume~\eqref{eqAssumptions4}, where $x := z+ku^1 + u^2 \in H^1\cap \Z^2$ minimizing $|k|$ satisfies $k \ge 2$.
Using the definition of {\bf Case~\hyperlink{case4}{4}}, we have $x-u^1 \not \in \intr(\GP(\uSet))$.
	%
	
	First, we show that $\flip(\uSet) \leq_r \uSet$.
	By Lemma~\ref{lemContainMin} we assume that $ \uSet$ does not contain a minimizer $z^*$ of~\eqref{eqMainProb}.
	By Lemma~\ref{Lemma:Gen_Gradient_Polyhedron}\ref{Prop:partI}, $z^* \in \intr(\GP(\uSet))$.
	If $z^* \in O_\uSet(z+u^1)$, then $z+u^1 \in \conv\{z^*, z, x\}$ yielding the contradiction $z+u^1 \in \intr(\GP(\uSet)$).
	If $z^* \in O_\uSet(z+u^2)$, then $z+u^2 \in \conv\{z^*, z, x\}$ yielding the contradiction $z+u^2 \in \GP(\uSet)$.
	If $z^* \in O_\uSet(z)$, then $z^* \in O_{\flip(\uSet)}(z) \cup O_{\flip(\uSet)}(z+u^1)$.
	This implies that either
	\(
	z \in \conv\{z^*, z+u^1, x\}
	\)
	and $z \in \intr(\GP(\uSet))$,	or
	\(
	z+u^1 \in \conv\{z^*, z, x+u^1\}
	\)
	and $z +u^1\in \intr(\GP(\uSet))$.
	Both conclusions are contradictions, so $z^* \not \in O_\uSet(z)$.	
	Thus, $z^* \in O_\uSet(z+u^1+u^2) $ and it can be written as
	\begin{equation}\label{eqCase4Opt}
	z^* = (z+u^1+u^2) +k_1u^1 + k_2u^2,
	\end{equation}
	where $k_1, k_2 \ge 0$ and $k_1 + k_2 \ge 1$.
	If $k_1 \le k_2$, then $z+u^1+u^2 \in \conv\{z^*, z, x\}$ yielding the contradiction $z+u^1+u^2 \in \GP(\uSet)$.
	Hence, $k_1 \ge k_2+1$.
	The inclusion $z^* \in O_\uSet(z+u^1+u^2)$ implies that $z^* \in O_{\flip(\uSet)}(x) \cup O_{\flip(\uSet)}(x+u^1)$ (see Figure~\ref{figCase4}).
	We consider both settings.

	If $z^* \in O_{\flip(\uSet)}(x)$, then write	
	\[
	z^* = x + (k(k_2+1)-k_1-1)(-u^1) + k_2 (ku^1+u^2).
	\]
	The vector $x-u^1$ can be written as
	\[
	\frac{k_2}{k(k_2+1)-k_1-1}\cdot z + \frac{k(k_2+1)-k_1-k_2-2}{k(k_2+1)-k_1-1}\cdot x  +\frac{1}{k(k_2+1)-k_1-1}\cdot z^*.
	\]
	Hence,  $k(k_2+1)-k_1-k_2 -2 < 0$ otherwise $x-u^1 \in \conv\{z, x, z^*\}$ yielding the contradiction $x-u^1 \in \intr(\GP(\uSet))$.
	This inequality is equivalent to $ k_1+k_2 > k(k_2+1)-2$.
	Thus,
	\[
	r_\uSet(z^*) = k_1+k_2 > k(k_2+1)-2 \ge k(k_2+1)-k_1 - 1 = r_{\flip(\uSet)}(z^*).
	\]
	Hence, $\flip(\uSet) \lr \uSet$.
	
	If $z^* \in O_{\flip(\uSet)}(x+u^1)$, then rearrange~\eqref{eqCase4Opt} to write $z^*$ as
	\[
	z^* = (x+u^1) + (k_1 - k_2k -k)u^1 + k_2 (ku_1+u_2).
	\]
	Thus,
	\[
	r_\uSet(z^*) = k_1+k_2 > (k_1-k_2k-k)+k_2 = r_{\flip(\uSet)}(z^*).
	\]
	Hence, $\flip(\uSet) \lr \uSet$.

Now, suppose that $\flip(\uSet) \ef \uSet$ and $\flip(\uSet) \er \uSet$.
Let $w^* \in \uSet$ minimize $f$ over $\uSet$.
Lemma~\ref{lemma:subset} and the equation $\flip(\uSet) \ef \uSet$ imply that $w^* \in \flip(\uSet) \cap \GP(\flip(\uSet))$.
	Note $x,x+u^1 \in \intr(\GP(\uSet))$ by the definition of \textbf{Case~\protect\hyperlink{case4}{4}}, so neither are cut by $z$ or $z+u^1$.
	Thus, $\GP(\flip(\uSet)) \cap \flip(\uSet)$ contains a least two vectors: $w^*$ and one of $x$ or $x+u^1$.
	The set $\flip(\uSet)$ can only be disconnected if $\GP(\flip(\uSet)) \cap \flip(\uSet) = \{z, x+u^1\}$ or $\GP(\flip(\uSet)) \cap\flip( \uSet)= \{z+u^1, x\}$.
	The first setting implies that $x+u^1$ strictly cuts $x$ and $z+u^1$ but not $z$, which is not possible by Lemma~\ref{Lemma:Gen_Gradient_Polyhedron}\ref{Prop:partIV}.
	The second setting implies that $x$ strictly cuts $x+u^1$ and $z$ but not $z+u^1$, which is again not possible.
	This proves \emph{(iii)}.
	\qed
	\end{proof}

\subsection{An analysis of {\bf Case 5}.}

We begin by showing
\begin{equation}\label{eqNotBothCase5}
\begin{array}{rl}
H^1 \cap \Z^2 = \emptyset~ &\text{or}~H^{-1} \cap \Z^2  = \emptyset;\\[.1 cm]
\text{if}~H^1 \cap \Z^2 \neq \emptyset, & \text{then } \{z-u^1+u^2\} = H^{1} \cap \Z^2, \text{ and}\\[.1 cm]
\text{if}~H^{-1} \cap \Z^2 \neq \emptyset, &\text{then } \{z+u^1-u^2\}= H^{-1} \cap \Z^2.
\end{array}
\end{equation}
The first statement in~\eqref{eqNotBothCase5} shows that the two outcomes in \textbf{Case~\hyperlink{case5}{5}} cannot both be satisfied.
Let $x:=z+ku^1-u^2 \in H^{-1} \cap \Z^2 $.
If $k \le 0$, then $z \in \conv\{x,z+u^1, z+u^2\}$ which contradicts $z \not\in \intr(\GP(\uSet))$.
If $k \ge 2$, then $z+u^1 \in \conv\{x,z+u^2, z\}$ which contradicts $z+u^1 \not\in \intr(\GP(\uSet))$.
Thus, $H^{-1} \cap \Z^2  \subseteq \{z+u^1-u^2\}$.
Similarly, $H^{1} \cap \Z^2 \subseteq \{z-u^1+u^2\}$.
If  $\{z+u^1-u^2\}= H^{-1} \cap \Z^2 $ and $\{z-u^1+u^2\}= H^{1} \cap \Z^2 $, then $z \in \conv\{z+u^1-u^2,z-u^1+u^2\} \subseteq \intr(\GP(\uSet))$, which is a contradiction.
This proves~\eqref{eqNotBothCase5}.

	\begin{lemma}\label{lemPart5ReduceROutcome}
	Suppose $\uSet$ fits into {\bf Case~\hyperlink{case5}{5}}.
	Then
	\begin{enumerate}[leftmargin = .75 cm, label = \textit{(\roman*)}]
	\item $\flip(\uSet) \ler \uSet$.
	\item if $\uSet$ does not contain a minimizer of~\eqref{eqMainProb}, then $\flip(\uSet) \lr \uSet$
	\item if $\flip(\uSet) \ef \uSet$ and $\flip(\uSet) \er \uSet$, then $\flip(\uSet)$ is connected and $\flip(\uSet) \cap \GP(\flip(\uSet)) \setminus \uSet \neq \emptyset$.
	\end{enumerate}
	\end{lemma}
	\begin{proof}
	First, we show that $\flip(\uSet) \leq_r \uSet$.
	By Lemma~\ref{lemContainMin} we assume that $ \uSet$ does not contain a minimizer of~\eqref{eqMainProb}.
	By Lemma~\ref{Lemma:Gen_Gradient_Polyhedron}\ref{Prop:partI}, $z^* \in \intr(\GP(\uSet))$.
	
	It was shown in~\eqref{eqNotBothCase5} that either $\flip(U) = (u^1-u^2, u^2)$ or $\flip(U) = (u^1, -u^1+u^2)$, depending on if $H^{-1} \cap \Z^2 \neq \emptyset$ or $H^1 \cap \Z^2 \neq \emptyset$, respectively.
These outcomes are symmetric, so we assume
\[
\flip(U) = (u^1-u^2, u^2) ~~\text{and}~~ \flip(\uSet) =\{z, z+u^1, z+u^2, z+u^1-u^2\}.
\]
Figure~\ref{figCase5} illustrates this setting.

			\begin{figure}
	\hspace*{-.25 cm}
	\begin{tabular}{c @{\hskip 1 cm}c}
	\begin{tikzpicture}[scale = .5]

	\node[] at (-1.5,1) {$O_\uSet(z)$};
	\node[black] at (5,-0.5){$O_\uSet(z+u^1)$};
	\node[black] at (-2,5){$O_\uSet(z+u^2)$};
	\node[] at (5,6.5){$O_\uSet(z+u^1+u^2)$};
	
	\clip (-.1,-.1) rectangle (7.1,6.1);
	\draw[red, opacity = .3, thick](1,7)--(7,-2);

	\draw[draw = none, fill = black!30, opacity = .5](-1,0)--(2,0)--(2,3)--(-1,3)--cycle;

	\draw[draw = none, fill = black!30, opacity = .5](3,0)--(3,3)--(7,3)--(7,0)--cycle;

	\draw[draw = none, fill = black!30, opacity = .5](-1,4)--(2,4)--(2,6)--(-1,6)--cycle;

	\draw[draw = none, fill = black!30, opacity = .5](3,4)--(3,6)--(7,6)--(7,4)--cycle;

	\draw[dashed, thick, color = black](2,3)--(3,3)--(3,4)--(2,4)--cycle;


	\foreach \i in {-1,...,7}{
	\foreach \j in {0,...,6}
	\draw[fill = black!50, draw = black!50](\i,\j) circle (.5 ex);
	}

		\draw[ draw = none, fill = red!30, opacity = .3](3,0)--(1.75, 3.75)--(2.225, 4.225)--(5,0)--cycle;
	\draw[red, thick](1,6)--(3,0);
	\draw[red, thick](0,2)--(4,6);
	\draw[red, thick](1,6)--(5,0);

	\draw[fill = black, draw = black](2,3) circle (.75 ex) node[black,  left = -1pt]{$z$};
	\draw[fill = black, draw = black](3,3) circle (.75 ex) node[black, right = -1pt]{$z+u^1$};
	\draw[fill = black, draw = black](2,4) circle (.75 ex) node[black, left = -1pt]{$z+u^2$};
	\draw[fill = black, draw = black](3,4) circle (.75 ex) node[black, above right = -2pt]{$z+u^1+u^2$};
	
	
	\end{tikzpicture}
	&

	\begin{tikzpicture}[scale = .5]

	\node[] at (0,-.5) {$O_{\flip(\uSet)}(z)$};
	\node[] at (6,6.5){$O_{\flip(\uSet)}(z+u^1)$};
	
	\clip (-.1,-.1) rectangle (7.1,6.1);

	\draw[draw = none, fill = black!30, opacity = .5](0,0)--(2,0)--(2,3)--(0,5)--cycle;

	\draw[draw = none, fill = black!30, opacity = .5](3,3)--(3,6)--(7,6)--(7,0) -- (6,0)--cycle;

	\draw[draw = none, fill = black!30, opacity = .5](2,4)--(2,6)--(0,6)-- cycle;

	\draw[draw = none, fill = black!30, opacity = .5](3,2)--(3,0)--(5,0) -- cycle;

	\draw[dashed, thick, color = black](2,3)--(2,4)--(3,3)--(3,2)--cycle;

	\foreach \i in {-1,...,7}{
	\foreach \j in {0,...,6}
	\draw[fill = black!50, draw = black!50](\i,\j) circle (.5 ex);
	}

	\draw[fill = black, draw = black](2,3) circle (.75 ex) node[black,  left = -1pt]{$z$};
	\draw[fill = black, draw = black](3,3) circle (.75 ex) node[black, above right = -1pt]{$z+u^1$};
	\draw[fill = black, draw = black](2,4) circle (.75 ex) node[black, left = -1pt]{$z+u^2$};
	\draw[fill = black, draw = black](3,2) circle (.75 ex) node[black, below right = -2pt]{$z+u^1-u^2$};

	\end{tikzpicture}
	\\
	\hspace{1.5 cm}\emph{(i)} & \emph{(ii)}
	\end{tabular}
	\caption{\emph{(i)} The orthants used to prove Theorem~\ref{thmMain1Again} in \textbf{Case~\protect\hyperlink{case5}{5}}.
	$\GP(\uSet)$ is also drawn.
	\emph{(ii)} The orthants of $\flip(\uSet) $ when $|H^{-1} \cap \GP(\uSet)| \ge 1$.
	}\label{figCase5}
	\end{figure}

If $z^* \in O_{\flip(\uSet)}(z)$, then $z \in \conv\{z^*, z+u^1-u^2, z+u^2\}$ yielding the contradiction $z \in \intr(\GP(\uSet))$.
Similarly, if $z^* \in O_{\flip(\uSet)}(z+u^2)$, then we obtain the contradiction $z+u^2 \in \intr(\GP(\uSet))$ and if $z^* \in O_{\flip(\uSet)}(z+u^1)$, then we obtain the contradiction $z+u^1 \in \intr(\GP(\uSet))$.
Hence, $z^* \in O_{\flip(\uSet)}(z+u^1-u^2) \subseteq O_{\uSet}(z+u^1)$.
We can write $z^*$ as
\[
\begin{array}{rcl}
z^* &=& (z+u^1) + k_1u^1 +k_2 (-u^2) \\[.1 cm]
& = & (z+u^1 - u^2) +  k_1 (u^1-u^2) + (k_2-k_1-1) (-u^2),
\end{array}
\]
where $k_1, k_2 \ge 0$ and $k_1+k_2 \ge 1$.
Note that $k_2-k_1-1 \ge 0$ because $z^* \in O_{\flip(\uSet)}(z+u^1-u^2)$.
Hence,
\[
r_\uSet(z^*) = k_1+k_2 > k_2 - 1 = r_{\flip(\uSet)}(z^*).
\]
This shows that $\flip(\uSet) \lr \uSet$.

Now, suppose that $\flip(\uSet) \ef \uSet$ and $\flip(\uSet) \er \uSet$.
Let $w^* \in \uSet$ minimize $f$ over $\uSet$.
Lemma~\ref{lemma:subset} and the equation $\flip(\uSet) \ef \uSet$ imply that $w^* \in \flip(\uSet) \cap \GP(\flip(\uSet))$.
Furthermore, because $z+u^1-u^2 \in H^{-1} \subseteq \intr(\GP(\uSet^i))$ it is not cut by $z, z+u^2$, or $z+u^1$.
	Thus, $ \flip(\uSet) \cap \GP( \flip(\uSet))$ contains at least two vectors: $w^*$ and $z+u^1 - u^2$.
	$ \flip(\uSet)$ can only be disconnected if $w^* = z+u^{2}$.
	However, $z+u^2$ does not strictly cut $z$ or $z+u^1$.
	Therefore, $z+u^1-u^2$ strictly cuts $z$ and $z+u^1$ but not $z+u^2$.
	This is not possible by Lemma~\ref{Lemma:Gen_Gradient_Polyhedron}~\ref{Prop:partIV}.
	Thus, $ \flip(\uSet)$ is connected and $z+u^1-u^2 \in \flip(\uSet) \cap \GP(\flip(\uSet)) \setminus \uSet$.
	\qed
\end{proof}
	
\section{Convergence proofs}


\subsection{Convergence to a minimum: The proof of Theorem~\ref{thmMain1}}

Recall that if $\uSet$ contains a vector $v$ with $\grad(v) = \mathbf{0}$, then we do not update $\uSet$.
The fact that each step uses constantly many gradient evaluations follows directly from counting the number of gradient evaluations in every case in Table~\ref{tableFlips}.
\textbf{Cases~\hyperlink{case3}{3}} to {\bf \hyperlink{case5}{5}} require the most evaluations: 12 to determine $\uSet \cap \GP(\uSet)$ and $8$ to determine $H^{-1}$ and $H^{1}$.
If $\uSet$ does not fit into the table, then it is lattice-free by Lemma~\ref{lemNoCase}.
Using the notation $\lef$ and $\ler$, we restate Theorem~\ref{thmMain1} \emph{(i)} and \emph{(ii)} as Theorem~\ref{thmMain1Again}.
The proof of Theorem~\ref{thmMain1Again} follows directly from Lemmata~\ref{lemPart1ReduceROutcome},~\ref{lemPart2ReduceROutcome},~\ref{lemPart3ReduceROutcome},~\ref{lemPart4ReduceROutcome}, and~\ref{lemPart5ReduceROutcome}.
\begin{theorem}\label{thmMain1Again}
Let $\uSet $ be preprocessed as in Lemma~\ref{eqPreprocess}.
If $\uSet$ fits into Table~\ref{tableFlips} and does not contain an optimal solution of~\eqref{eqMainProb}, then at least one of the following holds:
\begin{enumerate}[leftmargin=.85cm, label = {\it(\roman*)}]
\item\label{thmMain1AgainPart1} $\flip(\uSet) \lf \uSet$ and $\flip(\uSet)\ler \uSet$
\smallskip
\item\label{thmMain1AgainPart2}  $\flip(\uSet) \lef \uSet$ and $\flip(\uSet)\lr \uSet$
\smallskip
\item\label{thmMain1AgainPart3}  $\flip(\uSet) \lef \uSet$, $\flip(\uSet)\ler \uSet$, and $\flip(\uSet)$ is connected.
\end{enumerate}
If $\uSet$ is connected, then $\flip(\uSet)$ satisfies (ii).
\end{theorem}

Theorem~\ref{thmMain1Again} states that either $\lef$ or $\ler$ decreases after at most two updates.
Assumption~\eqref{eqLevelSets} implies that $\lef$ and $\ler$ can only decrease a finite number of times before $\uSet$ contains a minimizer of~\eqref{eqMainProb}.
\begin{theorem}\label{thm:converge_towards_minimum}
If~$\uSet^0$ is unimodular and $(\uSet^i)_{i=0}^\infty$ is constructed using Table~\ref{tableFlips}, then there exists an index $T_1$ such that ~$\uSet^{i}$ contains an optimal solution of~\eqref{eqMainProb} for all $i \ge T_1$.
\end{theorem}

We end this section with a convergence result.
Suppose $f$ is $L$-Lipschitz continuous and $c$-strongly convex.
We say $f$ is $c$-strongly convex for $c > 0$ if
\[
f(\overline{z}) \ge f(z) + \grad(z)^\intercal (\overline{z} - z) + c \cdot \|\overline{z} - z\|_2^2 \quad \forall ~ z, \overline{z} \in \R^2.
\]
Denote the identity matrix by $\mathbb{I}^2$.

\begin{proposition}\label{corConvergence}
Let $\uSet^0  = \uSet(z^0, U^0)$ be preprocessed as in Lemma~\ref{eqPreprocess} with $U^0 = \mathbb{I}^2$.
Let $z^* \in \Z^2$ be an optimal solution to~\eqref{eqMainProb}.
After $T_1 \le (6L/c+2) \cdot r_{\uSet^0}(z^*)$ many updates via Table~\ref{tableFlips}, the set $\uSet^{T_1}$ contains an optimal solution to~\eqref{eqMainProb}.
The set $\uSet^{T_1}$ is not necessarily lattice-free.
\end{proposition}

\begin{proof}
Let $i \ge 0$ and consider updating $\uSet^i = (z^i, U^i)$ to $\flip(\uSet^i) = \uSet^{i+1}$.
If $\uSet^i$ is connected, which occurs in \textbf{Cases~\hyperlink{case3}{3}} to {\bf \hyperlink{case5}{5}}, then $ \uSet^{i+1} \lr \uSet^i$ by Lemmata~\ref{lemPart3ReduceROutcome},~\ref{lemPart4ReduceROutcome}, and~\ref{lemPart5ReduceROutcome}.

Analysis of the disconnected cases uses the following argument.
If $z,w \in \Z^2$ are such that $z$ strictly cuts $w$, then by strong convexity we have
\[
f(w)  \ge f(z) + \grad(z)^\intercal (w-z) + c \cdot \|w - z\|_2^2
 > f(z)+c.
\]
Hence, the difference between $f(w)$ and $f(z)$ is at least $c$.

\textbf{Case~\hyperlink{case2}{2}} has three outcomes.
If the third outcome occurs, i.e., $\flip(U) = (-u^1, 2u^1+u^2)$, then $ \uSet^{i+1} \lr \uSet^i$ by Lemma~\ref{lemPart2ReduceROutcome}.
Otherwise, Lemma~\ref{lemCase2Extras} states that $\uSet^{i+1}$ is connected, in which case $ \uSet^{i+2} \lr \uSet^{i+1} \ler \uSet^i$, or the vector $w^*$ minimizing $f$ over $\uSet^i$ is strictly cut by some vector in $\uSet^{i+1}$, in which case
\begin{equation}\label{eqDropByC}
\min\{f(w) : w \in \uSet^{i+1} \} +c \le \min\{f(w) : w \in \uSet^{i} \}.
\end{equation}

Consider \textbf{Case~\hyperlink{case1}{1}}.
By Lemma~\ref{lemCase1Extras}, either $\uSet^{i+1}$ is connected, in which case $ \uSet^{i+2} \lr \uSet^{i+1} \ler \uSet^i$, or the vector $w^*$ minimizing $f$ over $\uSet^i$ is strictly cut by some vector in $\uSet^{i+1}$, in which case~\eqref{eqDropByC} holds.

Overall we have shown that either $\uSet^{i+2} \lr \uSet^{i}$ or~\eqref{eqDropByC} holds, independent of the case $\uSet^i$ falls into.
The value with respect to $\ler$ can decrease by at most $r_{\uSet^0}(z^*)$.
The minimum function value can decrease by at most the amount $\delta^f - f(z^*)$, where $w^* \in \uSet^0$ satisfies
\[
\delta^f :=f(w^*) = \min\{f(w) : w \in \uSet^0\}.
\]
Suppose $v^* \in \uSet^0$ satisfies $\|v^* - z^*\|_1 = r_{\uSet^0}(z^*)$.
Because $U^0 = \mathbb{I}^2$ we have $\|w^* - v^*\|_1 \le 2$ and
\[
\delta^f - f(z^*) = f(w^*) - f(z^*) \le L \cdot \|w^* - z^*\|_2 \le L \cdot \|w^* - z^*\|_1\le L \cdot (r_{\uSet^0}(z^*)+2).
\]
If $ r_{\uSet^0}(z^*) = 0$, then the lemma is trivially true.
If $ r_{\uSet^0}(z^*) \ge 1$, then $\uSet^i$ must contain a minimizer of~\eqref{eqMainProb} after at most
\[
2 \bigg[\frac{\delta^f - f(z^*)}{c}+r_{\uSet^0}(z^*)\bigg]
\le
2 \bigg[\frac{L \cdot (r_{\uSet^0}(z^*)+2)}{c}+ r_{\uSet^0}(z^*)\bigg] \le \bigg(\frac{6L}{c}+2\bigg)  r_{\uSet^0}(z^*)
\]
many updates.
\qed
\end{proof}

\subsection{Convergence towards a lattice-free set: The proof of Theorem~\ref{thmMain2}.}\label{subsecConvToLFree}

Let $\uSet^0$ be preprocessed as in Lemma~\ref{eqPreprocess}.
For $i \ge 0$ set $\uSet^{i+1} = \flip(\uSet^{i})$ and preprocess $\uSet^{i+1}$ via Lemma~\ref{eqPreprocess}.
By Theorem~\ref{thm:converge_towards_minimum} and Lemma~\ref{lemma:subset}, there exists $T_1 \ge 1$ such that $\uSet^i\cap\GP(\uSet^i)$ contains an optimal solution of~\eqref{eqMainProb} for all $i \ge T_1$.
Relabeling indices allows us to assume $T_1 = 0$.
After at most one flip each $\uSet^i$ is also connected.
\begin{lemma}\label{obs:Always_at_least_2}
	$\flip(\uSet^i)$ is connected for all $i \ge 1$.
\end{lemma}

\begin{proof}
	By Lemma~\ref{lemma:subset}, $\uSet^i$ contains a minimizer of~\eqref{eqMainProb} for all $i \ge 0$.
	Hence,	$\uSet^i \ef \uSet^{i+1} = 0$ and $\uSet^i \ef \uSet^{i+1} = 0$.
	Connectivity of $\uSet^{i+1}$ then follows from Lemmata~\ref{lemPart1ReduceROutcome},~\ref{lemPart2ReduceROutcome},~\ref{lemPart3ReduceROutcome},~\ref{lemPart4ReduceROutcome}, and~\ref{lemPart5ReduceROutcome}.
 \qed
\end{proof}

We finish with the proof of Theorem~\ref{thmMain2}.

\begin{proof}[of Theorem~\ref{thmMain2}]
By Theorem~\ref{thm:converge_towards_minimum} and  Lemma~\ref{obs:Always_at_least_2}, we may assume that $\uSet^i$ contains a minimizer $z^*$ of~\eqref{eqMainProb} and that $\uSet^i$ is connected for each $i \ge 0$.
For each $i \ge 0$ define the sets
\[
\begin{array}{rcl}
\cM^i_1&:=& \{w \in \uSet^i : f(w) = f(z^*)\}, ~~\text{and}\\[.1 cm]
\cM^i_j &:=& \{w \in \uSet^i : w~\text{minimizes}~f~\text{over}~ \uSet^i \cap \GP(\uSet^i) \setminus \cup_{k=1}^{j-1} \cM^i_j\} ~ \forall ~j \in \{2,3,4\}.
\end{array}
\]
By Lemma~\ref{lemma:subset} $z^* \in \cM^i_1 \subseteq \cM^{i+1}_1$.
There exists some $i$ such that $\cM^i_1 = \cM^{i+1}_1$ for all $i \ge 1$.
After reindexing, we assume $\cM^i_1 = \cM^{i+1}_1 \neq \emptyset$ for all $i \ge 1$.

Every time we update $\uSet^i$ via Table~\ref{tableFlips}, it fits into \textbf{Cases~\hyperlink{case3}{3}} to {\bf \hyperlink{case5}{5}}.
Parts \emph{(iii)} in Lemmata~\ref{lemPart3ReduceROutcome},~\ref{lemPart4ReduceROutcome} and~\ref{lemPart5ReduceROutcome} imply that there exists $ w \in \uSet^{i+1} \cap (\GP(\flip(\uSet^{i+1})) \setminus \uSet^i)$.
Note $w \not \in \cM^{i+1}_1$, otherwise $\cM^0_1 \subsetneq \cM^{i+1}_1$, and $w$ does not strictly cut any vector in $\cM^i_1$ by Proposition~\ref{propConvexProperties}.
Assumption~\eqref{eqLevelSets} then implies that there exists some $i \ge 1$ such that the function value defining $\cM^i_2$ is minimized.
Furthermore, there exists an index $i^* \ge 1$ such that $\cM^i_2 = \cM^{i+1}_2$ for all $i \ge i^*$.
After reindexing, we may assume $\cM^i_2 = \cM^{i+1}_2 \neq \emptyset$ for all $i \ge 0$.

If $\uSet^0$ still fits into the table, then we may repeat the previous process to assume $\cM^i_3 = \cM^{i+1}_3 \neq \emptyset$ for all $i \ge 0$.
Again, if $\uSet^0$ fits into the table, then we may repeat this one last time to assume $\cM^i_4 = \cM^{i+1}_4 \neq \emptyset$ for all $i \ge 0$.
Recall that we do not update $\uSet^i$ if it contains a vector $v$ with $\grad(v) = \mathbf{0}$ or it does not fit into the table, and in both situations $\GP(\uSet^i)$ is lattice-free by Lemma~\ref{lemNoCase}.
If we reach the step when $\cM^i_4 = \cM^{i+1}_4 \neq \emptyset$ for all $i \ge 0$, then the disjoint sets $\cM^0_1, \cM^0_2, \cM^3_0$ and $\cM^4_0$ are all non-empty.
These four sets are contained in $\uSet^0$.
Hence, $\GP(\uSet^0)$ is lattice-free by Corollary~\ref{Cor:4_points_lattice_free}.
\qed
\end{proof}
\section{Conclusions}
The results in this paper provide the first method for updating gradient polyhedra in dimension two.
The results extend to non-differentiable functions using subgradients.
Theorem~\ref{thmMain1Again} tells us that if we `flip' as defined in Table~\ref{tableFlips} until we no longer satisfy a case in the table, then we have an optimality certificate for \eqref{eqMainProb} in the form of a lattice-free gradient polyhedron.
This procedure mimics gradient descent in the following ways: each update is represented as a gradient polyhedron with at most $2^n$ facets which meets the bounds given by the theory, each update only needs constantly many gradient evaluations, and the initial set can be an arbitrary unimodular set.
There are many open questions.
For instance, it may be possible to adjust the `step size' of each flip to achieve faster convergence.
We believe this can be analyzed whenever $\uSet$ is such that multiple successive flips $\flip(\uSet), \flip(\flip(\uSet)), \dotsc$ fall into the same case.
Other questions include if the measure $\ler$ or the updates can be extended to $n \ge 3$ or to models with additional constraints.
%

\bibliographystyle{splncs04}
\bibliography{references}
\newpage

%
%
%
\end{document}